\documentclass[leqno,11pt,a4paper,twoside]{article}%
\usepackage{amssymb}
\usepackage{amsfonts}
\usepackage{amsmath}
\usepackage{graphicx}%
\providecommand{\U}[1]{\protect\rule{.1in}{.1in}}
\newtheorem{theorem}{Theorem}

\newtheorem{condition}[theorem]{Condition}

\newtheorem{definition}[theorem]{Definition}

\newtheorem{lemma}[theorem]{Lemma}

\newtheorem{proposition}[theorem]{Proposition}
\newtheorem{remark}[theorem]{Remark}

\newenvironment{proof}[1][Proof]{\noindent\textbf{#1.} }{\ \rule{0.5em}{0.5em}}
\begin{document}

\title{Asymptotic Control for a Class of Piecewise Deterministic Markov Processes
Associated to Temperate Viruses}
\author{Dan Goreac \thanks{Universit\'{e} Paris-Est, LAMA (UMR 8050), UPEMLV, UPEC,
CNRS, F-77454, Marne-la-Vall\'{e}e, France, Dan.Goreac@u-pem.fr}%
\thanks{\textbf{Acknowledgement.} This work has been partially supported by he
French National Research Agency project PIECE, number
\textbf{ANR-12-JS01-0006.}}}
\date{}
\maketitle

\begin{abstract}
We aim at characterizing the asymptotic behavior of value functions in the
control of piecewise deterministic Markov processes (PDMP) of switch type
under nonexpansive assumptions. For a particular class of processes inspired
by temperate viruses, we show that uniform limits of discounted problems as
the discount decreases to zero and time-averaged problems as the time horizon
increases to infinity exist and coincide. The arguments allow the limit value
to depend on initial configuration of the system and do not require
dissipative properties on the dynamics. The approach strongly relies on
viscosity techniques, linear programming arguments and coupling via random
measures associated to PDMP. As an intermediate step in our approach, we
present the approximation of discounted value functions when using piecewise
constant (in time) open-loop policies.

\end{abstract}



\textbf{Acknowledgement.} The author would like to thank the anonymous
referees for constructive remarks allowing to improve the manuscript.

\section{Introduction}

We focus on the study of some asymptotic properties in the control of a
particular family of piecewise deterministic Markov processes (abbreviated
PDMP), non diffusive, jump processes introduced in the seminal paper
\cite{Davis_84}. Namely, we are concerned with the existence of a limit of the
value functions minimizing the Ces\`{a}ro-type averages of some cost
functional as the time increases to infinity for controlled switch processes.
The main theoretical contribution of the paper is that the arguments in our
proofs are entirely independent on dissipativity properties of the PDMP and
they apply under mild nonexpansivity assumptions. Concerning the potential
applications, our systems are derived from the theory of stochastic gene
networks (and, in particular, genetic applets modelling temperate viruses).
Readers wishing to get acquainted to biological or mathematical aspects in
these models are referred to \cite{cook_gerber_tapscott_98},
\cite{hasty_pradines_dolnik_collins_00}, \cite{crudu_debussche_radulescu_09},
\cite{crudu_Debussche_Muller_Radulescu_2012}, \cite{G8}).

Switch processes can be described by a couple $\left(  \gamma_{\cdot}%
^{\gamma_{0},x_{0},u},X_{\cdot}^{\gamma_{0},x_{0},u}\right)  $, where the
first component is a pure jump process called mode and taking its values in
some finite set $\mathbb{M}$. The couple process is governed by a jump rate
and a transition measure, both depending on the current state of the system.
Between consecutive jumps, $X_{\cdot}^{\gamma_{0},x_{0},u}$ evolves according
to some mode-dependent flow. Finally, these characteristics (rate, measure,
flow) depend on an external control parameter $u$. Precise assumptions and
construction make the object of Section \ref{Section2}. In connection to these
jump systems, we consider the Abel-type (resp. Ces\`{a}ro-type) average
\[
\left.
\begin{array}
[c]{l}%
v^{\delta}\left(  \gamma_{0},x_{0}\right)  :=\underset{u}{\inf}\delta
\mathbb{E}\left[  \int_{0}^{\infty}e^{-\delta t}h\left(  \gamma_{t}%
^{\gamma_{0},x_{0},u},X_{t}^{\gamma_{0},x_{0},u},u_{t}\right)  dt\right]
,\text{ }\\
V_{T}\left(  \gamma_{0},x_{0}\right)  :=\underset{u}{\inf}\frac{1}%
{T}\mathbb{E}\left[  \int_{0}^{T}h\left(  \gamma_{t}^{\gamma_{0},x_{0}%
,u},X_{t}^{\gamma_{0},x_{0},u},u_{t}\right)  dt\right]  ,
\end{array}
\right.
\]
and investigate the existence of limits as the discount parameter
$\delta\rightarrow0,$ respectively the time horizon $T\rightarrow\infty.$

In the context of sequences of real numbers, the first result connecting
asymptotic behavior of Abel and Ces\`{a}ro means goes back to Hardy and
Littlewood in \cite{Hardy_Littlewood_1914}. Their result has known several
generalizations: to uncontrolled deterministic dynamics in \cite[XIII.5]%
{feller_1971}, to controlled deterministic systems in \cite{Arisawa_98_1},
\cite{Arisawa_98_2}, etc.

Ergodic behavior of systems and asymptotic of Ces\`{a}ro-type averages have
made the object of several papers dealing with either deterministic or
stochastic control systems. The partial differential system approach
originating in \cite{lions_papanicolaou_varadhan} relies on coercitivity of
the associated Hamiltonian (see also \cite{artstein_gaitsgory_2000} for
explicit criteria). Although the method generalizes to deterministic (resp.
Brownian) control systems in \cite{Arisawa_98_2} (resp.
\cite{arisawa_lions_stoc_98}), the main drawback resides in the fact that, due
to the ergodic setting, the limit is independent of the initial condition of
the control system. Another approach to the asymptotic behavior relies on
estimations on trajectories available under controllability and dissipativity
assumptions. The reader is referred to \ \cite{artstein_gaitsgory_2000},
\cite{bettiol_2005} for the deterministic setting or \cite{Basak_Borkar_Ghosh}%
, \cite{borkar_gaitsgory_07}, \cite{buckdahn_ichihara_2005},
\cite{richou_2009} for Brownian systems. Although the method is different, it
presents the same drawback as the PDE one: it fails to give general limit
value functions that depend on the initial data.

In the context of piecewise deterministic Markov processes, the
infinite-horizon optimal control literature is quite extensive
(\cite{Davis_86}, \cite{Soner86_2}, \cite{Dempster_Ye_96},
\cite{Almudevar_2001}, \cite{Forwick_Schal_Schmitz_2004}, etc.). To our best
knowledge, average control problems have first been considered in an impulsive
control framework in \cite{Costa_89} and \cite{Gatarek_92}. The first papers
dealing with long time average costs in the framework of continuous control
policies are \cite{Costa_Doufour_2010} and \cite{Costa_Dufour_2009} (see also
\cite{CostaDufour_Book2013}). The problem studied in the latter papers is
somewhat different, since it concerns $\underset{u}{\inf}%
\underset{T\rightarrow\infty}{\text{ }\lim\sup}\frac{1}{T}\mathbb{E}\left[
\int_{0}^{T}h\left(  \gamma_{t}^{\gamma_{0},x_{0},u},X_{t}^{\gamma_{0}%
,x_{0},u},u_{t}\right)  dt\right]  ,$ thus leading to an inf/sup formulation,
while, in our case, we deal with a sup/inf formulation. Moreover, the methods
employed are substantially different. Our work should be regarded as a
complement to the studies developed in the cited papers.

A nonexpansivity condition has been employed in
\cite{quincampoix_renault_2011} in connection to deterministic control systems
allowing to obtain the existence of a general (uniform) limit value function.
This method has been (partially) extended to Brownian control systems in
\cite{G2}. In both these papers, convenient estimates on the trajectories in
finite horizon allow to prove the uniform continuity of Ces\`{a}ro averages
$V_{T}$ and an intuition coming from repeated games theory (inspired by
\cite{renault_2011}) gives the candidate for the limit value function. If the
convergence to this limit value function is uniform, the results of
\cite{oliubarton:hal-00661833} for deterministic systems yield the equivalence
between Abel and Ces\`{a}ro long-time averages. This latter assertion is still
valid for controlled Brownian diffusions (see \cite[Theorems 10 and 13]{G2})
and (to some extent) for piecewise deterministic Markov processes (see
\cite[Theorem 4.1]{GoreacSerea_TauberianPDMP_2014}).

In the present paper, we generalize the results of
\cite{quincampoix_renault_2011} and \cite{G2} to the framework of switch
piecewise deterministic Markov processes. The methods are based on viscosity
solutions arguments. We deal with two specific problems. The key point is, as
for Brownian systems, a uniform continuity of average value functions with
respect to the average parameter ($\delta$ or $T$). However, the approach in
\cite{G2} benefits from dynamic programming principles, which, within the
framework of PDMP, are easier obtained for Abel means (discounted functions
$v^{\delta}$). This is why, results like \cite[Proposition 7 and Theorem
8]{G2} are not directly applicable and we cannot make use of the already
mentioned intuition on repeated games. To overcome this problem, we proceed as
follows : if the system admits an invariant compact set, we prove the uniform
continuity of $\left(  v^{\delta}\right)  _{\delta>0}$ and use the results in
\cite[Theorem 4.1]{GoreacSerea_TauberianPDMP_2014} to show (in Theorem
\ref{ThUniformConvMain}) that this family admits a unique adherent point with
respect to the topology of continuous functions and, hence, it converges
uniformly. This implies the existence of $\underset{T\rightarrow\infty}{\lim
}V_{T}\left(  \gamma_{0},x_{0}\right)  $ and the limit is uniform with respect
to the initial data.

The second problem is proving the uniform continuity of $\left(  v^{\delta
}\right)  _{\delta>0}$ under explicit noxexpansivity conditions. In the
Brownian setting, this follows from estimates on the trajectories and a
natural coupling with respect to the same Brownian motion in \cite[Lemma
3]{G2}. For switch PDMP, we obtain similar results (in a convenient setting)
by using some reference random measure generated by the process. Although the
second marginal of this coupling might not come from a controlled PDMP, it is
shown to belong to a convenient class of measures by using linear programming
techniques (developed in \cite{G8}, \cite{G7} and inspired by Krylov
\cite{krylov_00}).

Let us now explain how the paper is organized. We begin with fixing some
notations employed throughout the paper (in Subsection \ref{Subsection2.1}).
We proceed by recalling the construction of controlled PDMP of switch type and
present the main assumptions on the characteristics in Subsection
\ref{Subsection2.2}. In Subsection \ref{Subsection2.3}, we introduce the
concept of invariance with respect to PDMP dynamics, the value functions
(Ces\`{a}ro and Abel averages) and the occupation measures associated to
controlled dynamics (taken from \cite{G7}). The main contributions of our
paper are stated in Section \ref{Section3Main}. We begin with introducing a
very general, yet abstract nonexpansivity condition in Subsection
\ref{Subsection3.1AbstractCondition}, Condition \ref{NonExp}. The first main
result of the paper (Theorem \ref{ThUniformConvMain}) is given in Subsection
\ref{Subsection3.2Main1}. This result states that whenever the nonexpansivity
Condition \ref{NonExp} is satisfied, there exists a unique limit value
function $\underset{T\rightarrow\infty}{\lim}V_{T}=\underset{\delta
\rightarrow0}{\lim}v^{\delta}$ independent of the average considered (Abel/
Ces\`{a}ro). In subsection \ref{Subsection3.3ExplicitNonexp}, we give explicit
nonexpansive conditions on the dynamics and the cost functional implying
Condition \ref{NonExp}. The second main result of the paper (Subsection
\ref{Subsection3.4Main2}, Theorem \ref{ThMainFirst}) provides an explicit
construction of (pseudo-)couplings. 

We proceed with a biological framework justifying our models in Section
\ref{Section4Examples}. We present the foundations of Hasty's model for Phage
$\lambda$ inspired by \cite{hasty_pradines_dolnik_collins_00} in Subsection
\ref{Subsection4.1}. In order to give a hint to our readers less familiarized
with mathematical models in systems biology, we briefly explain how a PDMP can
be associated to Hasty's genetic applet in Subsection \ref{Subsection4.2}.
Finally, the aim of Subsection \ref{Subsection4.3} is to give an extensive
choice of characteristics satisfying all the assumptions of the main Theorem
\ref{ThMainFirst}.

Section \ref{Section5ProofMain1} gives the proof of the first main result
(Theorem \ref{ThUniformConvMain}). First, we prove that Condition \ref{NonExp}
implies the equicontinuity of the family of Abel-average value functions
$\left(  v^{\delta}\right)  _{\delta>0}.$ Next, we recall the results in
\cite{GoreacSerea_TauberianPDMP_2014} on Abel-type theorems to conclude. The
results of this section work in all the generality of \cite{Soner86_2} (see
also \cite{G8}, \cite{G7}).

The proof of the second main result (Theorem \ref{ThMainFirst}) is given in
Section \ref{Section6ProofMain2}. The proof is based on constructing explicit
couplings satisfying Condition \ref{NonExp} and it relies on four steps. The
first step is showing that the value functions $v^{\delta}$ can be suitably
approximated by using piecewise constant open-loop policies. This is done in
Subsection \ref{Subsection6.2PiecewiseCstePolicies}. The proof combines the
approach in \cite{krylov_00} with the dynamic programming principles in
\cite{Soner86_2}. We think that neither the result, nor the method are
surprising but, for reader's sake, we have provided the key elements in the
Appendix. The second step is to interpret the system as a stochastic
differential equation (SDE) with respect to some random measure (Subsection
\ref{Subsection6.3RandomMeasures}). The third step (Subsection
\ref{Subsection6.4MeasureEmbedding}) is to embed the solutions of these SDE in
a space of measures satisfying a suitable linear constraint via the linear
programming approach. To conclude, the fourth step (given in Subsection
\ref{Subsection6.5Coupling}) provides a constructive (pseudo-) coupling using
SDE estimates.

\section{Preliminaries\label{Section2}}

\subsection{Notations\label{Subsection2.1}}

Throughout the paper, we will use the following notations.

Unless stated otherwise, the Euclidian spaces $%
\mathbb{R}
^{N},$ for some $N\geq1$ are endowed with the usual Euclidian inner product
$\left\langle x,y\right\rangle =\left(  y^{t}\right)  x,$ where $y^{t}$ stands
for the transposed, row vector and with the associated norm $\left\vert
x\right\vert =\sqrt{\left\langle x,x\right\rangle },$ for all $x,y\in%
\mathbb{R}
^{N}$.

For every $r>0,$ the set $\overline{B}\left(  0,r\right)  $ denotes the closed
$r$-radius ball of $%
\mathbb{R}
^{N}.$

The set $\mathbb{M}$ will denote some finite set. Whenever needed, the set
$\mathbb{M}$ is endowed with the discrete topology.

Unless stated otherwise, $\mathbb{U}$ is a compact metric space referred to as
the control space. We let $\mathcal{A}_{0}\left(  \mathbb{U}\right)  $ denote
the space of $\mathbb{U}$-valued Borel measurable functions defined on
$\mathbb{M\times}%
\mathbb{R}
^{N}\times%
\mathbb{R}
_{+}$. The sequence $u=\left(  u_{1},u_{2},...\right)  ,$ where $u_{k}%
\in\mathcal{A}_{0}\left(  \mathbb{U}\right)  ,$ for all $k\geq1$ is said to be
an admissible control. The class of such sequences is denoted by
$\mathcal{A}_{ad}\left(  \mathbb{U}\right)  $ (or simply $\mathcal{A}_{ad}$
whenever no confusion is at risk concerning $\mathbb{U}$). We introduce, for
every $n\geq1,$  the spaces of piecewise constant policies%
\[
\left.
\begin{array}
[c]{l}%
\mathcal{A}_{0}^{n}\left(  \mathbb{U}\right)  =\left\{  u\in\mathcal{A}%
_{0}\left(  \mathbb{U}\right)  :\text{ }u\left(  \gamma,x,t\right)
=u^{0}\left(  \gamma,x\right)  1_{\left\{  0\right\}  }\left(  t\right)  +%
{\textstyle\sum\limits_{k\geq0}}
u^{k}\left(  \gamma,x\right)  1_{\left(  \frac{k}{n},\frac{k+1}{n}\right]
}\left(  t\right)  \right\}  ,\\
\mathcal{A}_{ad}^{n}\left(  \mathbb{U}\right)  =\left\{  \left(  u_{m}\right)
_{m\geq1}\in\mathcal{A}_{ad}\left(  \mathbb{U}\right)  :\text{ }u_{m}%
\in\mathcal{A}_{0}^{n}\left(  \mathbb{U}\right)  ,\text{ }m\geq1\right\}  .
\end{array}
\right.
\]
As before, we may drop the dependency on $\mathbb{U}$.

For every bounded function $\varphi:\mathbb{M\times}%
\mathbb{R}
^{N}\times\mathbb{U}\longrightarrow%
\mathbb{R}
^{k},$ for some $N,k\geq1$ which is Lipschitz-continuous with respect to the $%
\mathbb{R}
^{N}$ component, we let
\begin{align*}
\varphi_{\max} &  =\sup_{\left(  \gamma,x,u\right)  \in\mathbb{M\times}%
\mathbb{R}
^{N}\times\mathbb{U}}\left\vert \varphi_{\gamma}\left(  x,u\right)
\right\vert ,\text{ }Lip\left(  \varphi\right)  =\sup_{\substack{\left(
\gamma,x,y,u\right)  \in\mathbb{M\times}%
\mathbb{R}
^{2N}\times\mathbb{U}\\x\neq y}},\frac{\left\vert \varphi_{\gamma}\left(
x,u\right)  -\varphi_{\gamma}\left(  y,u\right)  \right\vert }{\left\vert
x-y\right\vert },\\
\text{ }\left\vert \varphi\right\vert _{1} &  =\varphi_{\max}+Lip\left(
\varphi\right)  .
\end{align*}
This is the Lipschitz norm of $\varphi.$

Whenever $\mathbb{K}\subset%
\mathbb{R}
^{N}$ is a closed set, we denote by $C\left(  \mathbb{M\times K};%
\mathbb{R}
\right)  $ the set of continuous real-valued functions defined on
$\mathbb{M\times K}$. \ The set $BUC\left(  \mathbb{M\times K};%
\mathbb{R}
\right)  $ stands for the family of real-valued bounded, uniformly continuous
functions defined on $\mathbb{M\times K}$.

The real-valued function $\varphi:%
\mathbb{R}
^{N}\longrightarrow%
\mathbb{R}
$ is said to be of class $C_{b}^{1}$ if it has continuous, bounded,
first-order derivatives. The gradient of such functions is denoted by
$\partial_{x}\varphi$.

The real-valued function $\varphi:\mathbb{M\times}%
\mathbb{R}
^{N}\longrightarrow%
\mathbb{R}
$ is said to be of class $C_{b}^{1}$ if $\varphi\left(  \gamma,\cdot\right)  $
is of class $C_{b}^{1},$ for all $\gamma\in\mathbb{M}.$

Given a generic metric space $\mathbb{A}$, we let $\mathcal{B\left(
\mathbb{A}\right)  }$ denote the Borel subsets of $\mathbb{A}$. We also let
$\mathcal{P}\left(  \mathbb{A}\right)  $ denote the family of probability
measures on $\mathbb{A}$. The distance $W_{1}$ is the usual Wasserstein
distance on $\mathcal{P}\left(  \mathbb{A}\right)  $ and $W_{1,Hausdorff}$ is
the usual Pompeiu-Hausdorff distance between subsets of $\mathcal{P}\left(
\mathbb{A}\right)  $ constructed with respect to $W_{1}$.

For a generic real vector space $\mathbb{A},$ we let $\overline{co}$ denote
the closed convex hull operator.

\subsection{Construction of Controlled Piecewise Deterministic Processes of
Switch Type\label{Subsection2.2}}

Piecewise deterministic Markov processes have been introduced in
\cite{Davis_84} and extensively studied for the last thirty years in
connection to various phenomena in biology (see \cite{cook_gerber_tapscott_98}%
, \cite{crudu_debussche_radulescu_09}, \cite{Wainrib_Thieullen_Pakdaman_2010},
\cite{crudu_Debussche_Muller_Radulescu_2012}, \cite{G8}), reliability or
storage modelling (in \cite{Boxma_Kaspi_Kella_Perry_2005},
\cite{DufourDutuitGonzalesZhang2008}), finance (in \cite{Rolski_Schmidli_2009}%
), communication networks (\cite{graham2009}), etc. The optimal control of
these processes makes the object of several papers (e.g. \cite{Davis_86},
\cite{Soner86_2}, \cite{Costa_Dufour_12}, etc.). For reader's sake we will
briefly recall the construction of these processes, the assumptions, as well
as the type of controls we are going to employ throughout the paper.

The switch PDMP is constructed on a space $\left(  \Omega,\mathcal{F}%
,\mathbb{P}\right)  $ allowing to consider a sequence of independent, $\left[
0,1\right]  $ uniformly distributed random variables (e.g. the Hilbert cube
starting from $\left[  0,1\right]  $ endowed with its Lebesgue measurable sets
and the Lebesgue measure for coordinate, see \cite[Section 23]{davis_93}). We
consider a compact metric space $\mathbb{U}$ referred to as the control space.
The process is given by a couple $\left(  \gamma,X\right)  ,$ where $\gamma$
is the discrete mode component and takes its values in some finite set
$\mathbb{M}$ and the state component $X$ takes its values in some Euclidian
state space $%
\mathbb{R}
^{N}$ ($N\geq1$)$.$ The process is governed by a characteristic triple :

- a family of bounded, uniformly continuous vector fields $f_{\gamma}:%
\mathbb{R}
^{N}\times\mathbb{U}\longrightarrow%
\mathbb{R}
^{N}$ such that $\left\vert f_{\gamma}\left(  x,u\right)  -f_{\gamma}\left(
y,u\right)  \right\vert \leq C\left\vert x-y\right\vert ,$ for some $C>0$ and
all $x,y\in%
\mathbb{R}
^{N},$ $\gamma\in\mathbb{M}$ and all $u\in\mathbb{U}$,

- a family of bounded, uniformly continuous jump rates $\lambda_{\gamma}:%
\mathbb{R}
^{N}\times\mathbb{U}\longrightarrow%
\mathbb{R}
_{+}$ such that $\left\vert \lambda_{\gamma}\left(  x,u\right)  -\lambda
_{\gamma}\left(  y,u\right)  \right\vert \leq C\left\vert x-y\right\vert ,$
for some $C>0$ and all $x,y\in%
\mathbb{R}
^{N},$ $\gamma\in\mathbb{M}$ and all $u\in\mathbb{U}$,

- a transition measure $Q:\mathbb{M}\times%
\mathbb{R}
^{N}\times\mathbb{U}\longrightarrow\mathcal{P}\left(  \mathbb{M\times}%
\mathbb{R}
^{N}\right)  $. We assume that this transition measure has the particular
form
\[
Q\left(  \gamma,x,u,d\theta dy\right)  =\delta_{x+g_{\gamma}\left(
\theta,x,u\right)  }\left(  dy\right)  Q^{0}\left(  \gamma,u,d\theta\right)
,
\]
for all $\left(  \gamma,x,u\right)  \in\mathbb{M}\times%
\mathbb{R}
^{N}\times\mathbb{U}$. The bounded, uniformly continuous jump functions
$g_{\gamma}:\mathbb{M}\times%
\mathbb{R}
^{N}\times\mathbb{U}\longrightarrow%
\mathbb{R}
^{N}$ are such that $\left\vert g_{\gamma}\left(  \theta,x,u\right)
-g_{\gamma}\left(  \theta,y,u\right)  \right\vert \leq C\left\vert
x-y\right\vert ,$ for some $C>0$ and all $x,y\in%
\mathbb{R}
^{N},$ all $\left(  \theta,\gamma\right)  \in\mathbb{M}^{2}$ and all
$u\in\mathbb{U}$. The transition measure for the mode component is given by
$Q^{0}:\mathbb{M}\times\mathbb{U}\longrightarrow\mathcal{P}\left(
\mathbb{M}\right)  .$ For every $A\subset\mathbb{M},$ the function $\left(
\gamma,u\right)  \mapsto Q^{0}\left(  \gamma,u,A\right)  $ is assumed to be
measurable and, for every $\left(  \gamma,u\right)  \in\mathbb{M}%
\times\mathbb{U}$, $Q^{0}\left(  \gamma,u,\left\{  \gamma\right\}  \right)
=0.$

These assumptions are needed in order to guarantee smoothness of value
functions in this context (see also \cite{G8} for further comments). Of
course, more general transition measures $Q$ can be considered under the
assumptions of \cite{G8}, \cite{G7} and the results of Subsection
\ref{Subsection3.2Main1} still hold true. However, the approach in Section
\ref{Section6ProofMain2} only holds true for these particular dynamics and it
is the reason why we have chosen to work under these conditions. Whenever
$u\in\mathcal{A}_{0}\left(  \mathbb{U}\right)  $ and $\left(  t_{0},\gamma
_{0},x_{0}\right)  \in%
\mathbb{R}
_{+}\times\mathbb{M\times}%
\mathbb{R}
^{N},$ we consider the ordinary differential equation%
\[
\left\{
\begin{array}
[c]{l}%
d\Phi_{t}^{t_{0},x_{0},u;\gamma_{0}}=f_{\gamma_{0}}\left(  \Phi_{t}%
^{t_{0},x_{0},u;\gamma_{0}},u\left(  \gamma_{0},x_{0},t-t_{0}\right)  \right)
dt,\text{ }t\geq t_{0},\\
\Phi_{t_{0}}^{t_{0},x_{0},u;\gamma_{0}}=x_{0.}%
\end{array}
\right.
\]

Given some sequence $u:=\left(  u_{1},u_{2},...\right)  \subset\mathcal{A}%
_{0}\left(  \mathbb{U}\right)  ,$ the first jump time $T_{1}$ has a jump rate
$\lambda_{\gamma_{0}}\left(  \Phi_{t}^{0,x_{0},u_{1};\gamma_{0}},u_{1}\left(
\gamma_{0},x_{0},t\right)  \right)  $, i.e. $\mathbb{P}\left(  T_{1}\geq
t\right)  =\exp\left(  -\int_{0}^{t}\lambda_{\gamma_{0}}\left(  \Phi
_{s}^{0,x_{0},u_{1};\gamma_{0}},u_{1}\left(  \gamma_{0},x_{0},s\right)
\right)  ds\right)  .$ The controlled PDMP is defined by setting $\left(
\Gamma_{t}^{\gamma_{0},x_{0},u},X_{t}^{\gamma_{0},x_{0},u}\right)  =\left(
\gamma_{0},\Phi_{t}^{0,x_{0},u_{1};\gamma_{0}}\right)  ,$ if $t\in\left[
0,T_{1}\right)  .$ The post-jump location $\left(  \Upsilon_{1},Y_{1}\right)
$ has $Q\left(  \gamma_{0},\Phi_{\tau}^{0,x_{0},u_{1};\gamma_{0}},u_{1}\left(
\gamma_{0},x_{0},\tau\right)  ,\cdot\right)  $ as conditional distribution
given $T_{1}=\tau.$ Starting from $\left(  \Upsilon_{1},Y_{1}\right)  $ at
time $T_{1}$, we select the inter-jump time $T_{2}-T_{1}$ such that
\[
\mathbb{P}\left(  T_{2}-T_{1}\geq t\text{ }/\text{ }T_{1},\Upsilon_{1}%
,Y_{1}\right)  =\exp\left(  -\int_{T_{1}}^{T_{1}+t}\lambda_{\Upsilon_{1}%
}\left(  \Phi_{s}^{T_{1},Y_{1},u_{2};\Upsilon_{1}},u_{2}\left(  \Upsilon
_{1},Y_{1},s-T_{1}\right)  \right)  ds\right)  .
\]
We set $\left(  \Gamma_{t}^{\gamma_{0},x_{0},u},X_{t}^{\gamma_{0},x_{0}%
,u}\right)  =\left(  \Upsilon_{1},\Phi_{t}^{T_{1},Y_{1},u_{2};\Upsilon_{1}%
}\right)  ,$ if $t\in\left[  T_{1},T_{2}\right)  .$ The post-jump location
$\left(  \Upsilon_{2},Y_{2}\right)  $ satisfies%
\[
\mathbb{P}\left(  \left(  \Upsilon_{2},Y_{2}\right)  \in A\text{ }/\text{
}T_{2},T_{1},\Upsilon_{1},Y_{1}\right)  =Q\left(  \Upsilon_{1},\Phi_{T_{2}%
}^{T_{1},Y_{1},u_{2};\Upsilon_{1}},u_{2}\left(  \Upsilon_{1},Y_{1},T_{2}%
-T_{1}\right)  ,A\right)  ,
\]
for all Borel set $A\subset\mathbb{M\times}%
\mathbb{R}
^{N}.$ And so on. For simplicity purposes, we set $\left(  T_{0},\Upsilon
_{0},Y_{0}\right)  =\left(  0,\gamma_{0},x_{0}\right)  $.

\subsection{Definitions\label{Subsection2.3}}

Before stating the main assumptions and results of our paper we will need to
recall some concepts : invariance with respect to PDMP dynamics, the Abel and
Ces\`{a}ro value functions and the embedding of trajectories into occupation measures.

\subsubsection{Invariance\label{2.3.Invariance} }

In order to get convenient estimates on the trajectories, we assume, whenever
necessary, that the switch system admits some invariant compact set
$\mathbb{K}.$ For the applications we have in mind, this is not a drawback
since, for biological systems, we deal either with discrete components or with
normalized concentrations (hence not exceeding given limits).\ We recall the
notion of invariance.

\begin{definition}
The closed set $\mathbb{K}$ is said to be invariant with respect to the
controlled PDMP with characteristics $\left(  f,\lambda,Q\right)  $ if, for
every $\left(  \gamma,x\right)  \in\mathbb{M\times K}$ and every
$u\in\mathcal{A}_{ad},$ one has $X_{t}^{\gamma,x,u}\in\mathbb{K},$ for all
$t\geq0,$ $\mathbb{P-}a.s.$
\end{definition}

Explicit geometric conditions on the coefficients and the normal cone to
$\mathbb{K}$ equivalent to the property of invariance are given in
\cite[Theorem 2.8]{G8}. Roughly speaking, these properties are derived from
the sub/superjet formulation of the condition on $d_{\mathbb{K}}$ being a
viscosity supersolution of some associated Hamilton-Jacobi integrodifferential
system. This invariance condition is natural even for purely deterministic
nonexpansive systems (cf. \cite{quincampoix_renault_2011}). It can be avoided
either by localization procedures or by imposing some relative compactness on
reachable sets (or, equivalently, occupation measures). We prefer to work
under this condition in order to focus on specific details of our method,
rather than localization technicalities.

\subsubsection{Value Functions\label{2.3.Value}}

We investigate the asymptotic behavior of discounted value functions (also
known as Abel-averages)
\begin{align*}
v^{\delta}\left(  \gamma,x\right)   &  :=\inf_{u\in\mathcal{A}_{ad}}%
\delta\mathbb{E}\left[  \int_{0}^{\infty}e^{-\delta t}h\left(  \Gamma
_{t}^{\gamma,x,u},X_{t}^{\gamma,x,u}\right)  dt\right]  \\
&  =\inf_{u=\left(  u_{n}\right)  _{n\geq1}\in\mathcal{A}_{ad}}\delta
\mathbb{E}\left[  \sum_{n\geq1}\int_{T_{n-1}}^{T_{n}}e^{-\delta t}h\left(
\Gamma_{t}^{\gamma,x,u},X_{t}^{\gamma,x,u},u_{n}\left(  \Gamma_{T_{n-1}%
}^{\gamma,x,u},X_{T_{n-1}}^{\gamma,x,u},t-T_{n-1}\right)  \right)  dt\right]
,
\end{align*}
$\gamma\in\mathbb{M},$ $x\in%
\mathbb{R}
^{N},$ $\delta>0,$ as the discount parameter $\delta\rightarrow0$ and
Ces\`{a}ro-average values%
\begin{align*}
V_{t}\left(  \gamma,x\right)   &  :=\inf_{u\in\mathcal{A}_{ad}}\frac{1}%
{t}\mathbb{E}\left[  \int_{0}^{t}h\left(  \Gamma_{s}^{\gamma,x,u}%
,X_{s}^{\gamma,x,u},u_{s}\right)  ds\right]  \\
&  =\inf_{u=\left(  u_{n}\right)  _{n\geq1}\in\mathcal{A}_{ad}}\frac{1}%
{t}\mathbb{E}\left[  \sum_{n\geq1}\int_{T_{n-1}\wedge t}^{T_{n}\wedge
t}h\left(  \Gamma_{s}^{\gamma,x,u},X_{s}^{\gamma,x,u},u_{n}\left(
\Gamma_{T_{n-1}}^{\gamma,x,u},X_{T_{n-1}}^{\gamma,x,u},s-T_{n-1}\right)
\right)  ds\right]  ,
\end{align*}
$\gamma\in\mathbb{M},$ $x\in%
\mathbb{R}
^{N},$ $t>0,$ as the time horizon $t\rightarrow\infty$. The cost function
$\ h:\mathbb{M}\times%
\mathbb{R}
^{N}\times\mathbb{U}\longrightarrow%
\mathbb{R}
$ is assumed to be bounded, uniformly continuous and Lipschitz continuous
w.r.t. the state component, uniformly in control and mode (i.e. $\left\vert
h\left(  \gamma,x,u\right)  -h\left(  \gamma,y,u\right)  \right\vert \leq
C\left\vert x-y\right\vert ,$ for some $C>0$ and all $\gamma\in\mathbb{M},$
$x,y\in%
\mathbb{R}
^{N}$ and all $u\in\mathbb{U}$).

\subsubsection{The Infinitesimal Generator and Occupation
Measures\label{2.3.Occupation}}

We recall that, for regular functions $\varphi$ (for example of class
$C_{b}^{1}$), the generator of the control \ process is given by
\[
\mathcal{L}^{u}\varphi\left(  \gamma,x\right)  =\left\langle f_{\gamma}\left(
x,u\right)  ,\partial_{x}\varphi\left(  \gamma,x\right)  \right\rangle
+\lambda_{\gamma}\left(  x,u\right)  \int_{\mathbb{M}}\left(  \varphi\left(
\theta,x+g_{\gamma}\left(  \theta,x,u\right)  \right)  -\varphi\left(
\gamma,x\right)  \right)  Q^{0}\left(  \gamma,u,d\theta\right)  ,
\]
for all $\left(  \gamma,x\right)  \in\mathbb{M\times}\mathcal{%
\mathbb{R}
}^{N}$ and any $u\in\mathcal{A}_{ad}\left(  \mathbb{U}\right)  $. A complete
description of the domain of this operator can be found, for instance, in
\cite[Theorem 26.14]{davis_93}.

To any $\left(  \gamma,x\right)  \in\mathbb{M\times}\mathcal{%
\mathbb{R}
}^{N}$ and any $u\in\mathcal{A}_{ad}\left(  \mathbb{U}\right)  $, we associate
the discounted occupation measure%
\begin{equation}
\mu_{\gamma,x,u}^{\delta}\left(  A\right)  =\delta\mathbb{E}\left[  \int%
_{0}^{\infty}e^{-\delta t}1_{A}\left(  \Gamma_{t}^{\gamma,x,u},X_{t}%
^{\gamma,x,u},u_{t}\right)  dt\right]  ,\label{gammauC}%
\end{equation}
for all Borel subsets $A\subset\mathbb{M\times}\mathcal{%
\mathbb{R}
}^{N}\times\mathbb{U}.$ The set of all discounted occupation measures is
denoted by $\Theta_{0}^{\delta}\left(  \gamma,x\right)  .$ We also define
\begin{equation}
\Theta^{\delta}\left(  \gamma,x\right)  =\left\{
\begin{array}
[c]{c}%
\mu\in\mathcal{P}\left(  \mathbb{M\times}\mathcal{%
\mathbb{R}
}^{N}\times\mathbb{U}\right)  \text{ }s.t.\text{ }\forall\phi
:\mathbb{M\longrightarrow}C_{b}^{1}\left(
\mathbb{R}
^{N}\right)  ,\\
\int_{\mathbb{M\times}\mathcal{%
\mathbb{R}
}^{N}\times\mathbb{U}}\left(  \mathcal{L}^{u}\phi\left(  \theta,y\right)
+\delta\left(  \phi(\gamma,x)-\phi\left(  \theta,y\right)  \right)  \right)
\mu\left(  d\theta,dy,du\right)  =0
\end{array}
\right\}  .\label{Thetadelta}%
\end{equation}
Links between $\Theta_{0}^{\delta}\left(  \gamma,x\right)  $ and
$\Theta^{\delta}\left(  \gamma,x\right)  $ will be given in Theorem
\ref{ThLP}. For further details, the reader is referred to \cite{G7}.

\section{Assumptions and Main Results\label{Section3Main}}

In this section, we present the main assumptions and results of our paper.

We begin with giving an abstract nonexpansivity condition under which the Abel
means $\left(  v^{\delta}\right)  _{\delta>0}$ and the Ces\`{a}ro means
$\left(  V_{t}\right)  _{t>0}$ converge uniformly and to a common limit. It is
a very general one and, in a less general form, it reads "a coupling $\mu$ can
be found between a fixed controlled trajectory starting from $x$ and another
one starting from $y$ such that the difference of costs evaluated on the two
trajectories be controlled by the distance $\left\vert x-y\right\vert $". This
is essential in proving ergodic behavior. In the uncontrolled dissipative case
(see, for example \cite{Benaim_LEBorgne_Malrieu_Zitt_2012}), the couplings are
such that the distance between the law of $X_{t}^{x}$ and the one of
$X_{t}^{y}$ decreases exponentially (is upper-bounded by some term
$e^{-ct}\left\vert x-y\right\vert $)$.$ In particular, this implies that
$\underset{\delta\rightarrow0}{\lim}v^{\delta}\left(  x\right)  $ is constant
(independent of $x$). Unlike the classical dissipative approach, our framework
allows the limit to depend on the initial data. The first main result of the
paper states that, under the nonexpansivity Condition \ref{NonExp}, the Abel
means $\left(  v^{\delta}\right)  _{\delta>0}$ and the Ces\`{a}ro means
$\left(  V_{t}\right)  _{t>0}$ converge uniformly and to a common limit.

The main drawback of this Condition \ref{NonExp} is that it is abstract
(theoretical). In a setting inspired by gene networks, we give an explicit
condition (Condition \ref{NonExpCondition}) on the characteristics of the PDMP
implying the abstract nonexpansivity condition. The second main result of the
paper allows to link the explicit Condition \ref{NonExpCondition} and the
abstract one given before. This is done by constructing suitable
(pseudo-)couplings and the proof requires several steps. 

\subsection{An Abstract Nonexpansivity
Condition\label{Subsection3.1AbstractCondition}}

Throughout the section, we assume the following nonexpansivity condition

\begin{condition}
\label{NonExp}For every $\delta>0,$ every $\varepsilon>0,$ every $\left(
\gamma,x,y\right)  \in\mathbb{M}\times%
\mathbb{R}
^{2N}$ and every $u\in\mathcal{A}_{ad}\left(  \mathbb{U}\right)  ,$ there
exists $\mu\in\mathcal{P}\left(  \left(  \mathbb{M}\times%
\mathbb{R}
^{N}\times\mathbb{U}\right)  ^{2}\right)  $ such that
\[
\left.
\begin{array}
[c]{l}%
i.\text{ \ \ }\mu\left(  \cdot,\mathbb{M}\times%
\mathbb{R}
^{N}\times\mathbb{U}\right)  =\mu_{\gamma,x,u}^{\delta}\in\Theta_{0}^{\delta
}\left(  \gamma,x\right)  ;\\
ii.\text{ \ }\mu\left(  \mathbb{M}\times%
\mathbb{R}
^{N}\times\mathbb{U},\cdot\right)  \in\Theta^{\delta}\left(  \gamma,y\right)
;\\
iii.\text{ }\underset{\left(  \mathbb{M}\times%
\mathbb{R}
^{N}\times\mathbb{U}\right)  ^{2}}{\int}\left\vert h\left(  \theta,z,w\right)
-h\left(  \theta^{\prime},z^{\prime},w^{\prime}\right)  \right\vert \mu\left(
d\theta,dz,dw,d\theta^{\prime},dz^{\prime},dw^{\prime}\right)  \leq Lip\left(
h\right)  \left\vert x-y\right\vert +\varepsilon.
\end{array}
\right.
\]

\end{condition}

\begin{remark}
Whenever $h$ only depends on the $x$ component (but not on the mode $\gamma,$
nor on the control $u$), one can impose
\[
W_{1}\left(  \widetilde{\Theta}_{0}^{\delta}\left(  \gamma,x\right)
,\widetilde{\Theta}^{\delta}\left(  \gamma,y\right)  \right)  \leq\left\vert
x-y\right\vert ,
\]
where $\widetilde{\Theta}_{0}$ (resp. $\widetilde{\Theta}$) denote the
marginals $\mu\left(  \mathbb{M},\cdot,\mathbb{U}\right)  $ of measures
$\mu\in$ $\Theta_{0}$ (resp. $\Theta)$. One can impose the slightly stronger
conditions
\[
W_{1}\left(  \widetilde{\Theta}_{0}^{\delta}\left(  \gamma,x\right)
,\widetilde{\Theta}_{0}^{\delta}\left(  \gamma,y\right)  \right)
\leq\left\vert x-y\right\vert \text{ or }W_{1,Hausdorff}\left(
\widetilde{\Theta}^{\delta}\left(  \gamma,x\right)  ,\widetilde{\Theta
}^{\delta}\left(  \gamma,y\right)  \right)  \leq\left\vert x-y\right\vert
\]
and the notion of nonexpansivity is transparent in this setting.
\end{remark}

\subsection{First Main Result (Existence of Limit Values and Abel-Tauberian
Results)\label{Subsection3.2Main1}}

The first main result of the paper states that, under the nonexpansivity
Condition \ref{NonExp}, the Abel means $\left(  v^{\delta}\right)  _{\delta
>0}$ and the Ces\`{a}ro means $\left(  V_{t}\right)  _{t>0}$ converge
uniformly and to a common limit.

\begin{theorem}
\label{ThUniformConvMain}Let us assume that there exists a compact set
$\mathbb{K}\subset%
\mathbb{R}
^{N}$ invariant with respect to the piecewise deterministic dynamics.
Moreover, we assume Condition \ref{NonExp} to hold true for every $\left(
\gamma,x,y\right)  \in\mathbb{M}\times\mathbb{K}^{2}$. Then, $\left(
v^{\delta}\right)  _{\delta>0}$ admits a unique limit $v^{\ast}\in C\left(
\mathbb{M\times K};%
\mathbb{R}
\right)  $ and $\left(  V_{t}\right)  _{t>0}$ converges to $v^{\ast}$ uniformly.
\end{theorem}

The proof is postponed to Section \ref{Section5ProofMain1}. In order to prove
Theorem \ref{ThUniformConvMain}, we proceed as follows. First, we prove that
Condition \ref{NonExp} implies the equicontinuity of the family of
Abel-average value functions $\left(  v^{\delta}\right)  _{\delta>0}.$ Next,
we recall the results in \cite{GoreacSerea_TauberianPDMP_2014} on Abel-type
theorems to conclude.

\subsection{An Explicit Nonexpansive
Framework\label{Subsection3.3ExplicitNonexp}}

For the remaining of the section, we assume that the control is given by a
couple $\left(  u,v\right)  \in\mathbb{U}:=U\times V$ acting as follows : the
jump rate (and the measure $Q^{0}$ giving the new mode) only depend on the
mode component and is controlled by the parameter $u.$ The component $X$ is
controlled both by $u$ and by $v$ and it behaves as in the general case. One
has a vector field $f:\mathbb{M}\times%
\mathbb{R}
^{N}\times U\times V\longrightarrow%
\mathbb{R}
^{N}$, a jump rate $\lambda:\mathbb{M}\times%
\mathbb{R}
^{N}\times U\times V\rightarrow%
\mathbb{R}
_{+}$ given by
\[
\lambda_{\gamma}\left(  x,u,v\right)  =\lambda_{\gamma}\left(  u\right)  ,
\]
and the transition measure $Q:\mathbb{M}\times%
\mathbb{R}
^{N}\times U\times V\rightarrow\mathcal{P}\left(  \mathbb{M}\times%
\mathbb{R}
^{N}\right)  $ having the particular form%
\[
Q\left(  \left(  \gamma,x\right)  ,u,v,d\theta dy\right)  =\delta
_{x+g_{\gamma}\left(  \theta,x,u,v\right)  }\left(  dy\right)  Q^{0}\left(
\gamma,u,d\theta\right)  ,
\]
where $Q^{0}$ governs the post-jump position of the mode component. In this
case, the extended generator of $\left(  \gamma,X\right)  $ has the form%
\begin{equation}
\mathcal{L}^{u,v}\varphi\left(  \gamma,x\right)  =\left\langle f_{\gamma
}\left(  x,u,v\right)  ,\partial_{x}\varphi\left(  \gamma,x\right)
\right\rangle +\lambda\left(  \gamma,u\right)  \int_{\mathbb{M}}\left(
\varphi\left(  \theta,x+g_{\gamma}\left(  \theta,x,u,v\right)  \right)
-\varphi\left(  \gamma,x\right)  \right)  Q^{0}\left(  \gamma,u,d\theta
\right)  .\label{Luv}%
\end{equation}

We will show that the results on convergence of the discounted value functions
hold true under the following explicit condition on the dynamics.

\begin{condition}
\label{NonExpCondition}For every $\gamma\in\mathbb{M}$, every $u\in U$ and
every $x,y\in%
\mathbb{R}
^{N},$ the following holds true%
\[
\sup_{v\in V}\inf_{w\in V}\max\left\{
\begin{array}
[c]{l}%
\text{ \ \ \ \ }\left\langle f_{\gamma}\left(  x,u,v\right)  -f_{\gamma
}\left(  y,u,w\right)  ,x-y\right\rangle ,\\
\underset{\theta\in\mathbb{M}}{\sup}\left\vert x+g_{\gamma}\left(
\theta,x,u,v\right)  -y-g_{\gamma}\left(  \theta,y,u,w\right)  \right\vert
-\left\vert x-y\right\vert ,\\
\text{ \ \ \ \ }\left\vert h\left(  \gamma,x,u,v\right)  -h\left(
\gamma,y,u,w\right)  \right\vert -Lip\left(  h\right)  \left\vert
x-y\right\vert
\end{array}
\right\}  \leq0.
\]

\end{condition}

\begin{remark}
1. Whenever $h$ does not depend on the control, the latter condition naturally
follows from the Lipschitz-continuity of $h$.

2. If, moreover, the post-jump position is given by a (state and control free)
translation $x\mapsto x+g_{\gamma}\left(  \theta\right)  ,$ this condition is
the usual deterministic nonexpansive one (i.e.
\[
\sup_{v\in V}\inf_{w\in V}\left\langle f_{\gamma}\left(  x,u,v\right)
-f_{\gamma}\left(  y,u,w\right)  ,x-y\right\rangle \leq0\text{ ).}%
\]
This kind of jump (up to a slight modification guaranteeing that protein
concentrations do not become negative) fits the general theory described in
\cite{crudu_debussche_radulescu_09}.
\end{remark}

\subsection{Second Main Result (Explicit Coupling)\label{Subsection3.4Main2}}

The second main result of the paper allows to link the explicit Condition
\ref{NonExpCondition} and the abstract Condition \ref{NonExp} in the framework
described in Subsection \ref{Subsection3.3ExplicitNonexp}.

\begin{theorem}
\label{ThMainFirst} We assume Condition \ref{NonExpCondition} to hold true.
Moreover, we assume that there exists a compact set $\mathbb{K}$ invariant
with respect to the PDMP governed by $\left(  f,\lambda,Q\right)  $. Then the
conclusion of Theorem \ref{ThUniformConvMain} holds true (i.e. the family
$\left(  v^{\delta}\right)  _{\delta>0}$ admits a unique limit $v^{\ast}\in
C\left(  \mathbb{M\times K};%
\mathbb{R}
\right)  $ and $\left(  V_{t}\right)  _{t>0}$ converges to $v^{\ast}$ uniformly).
\end{theorem}

The proof of this result is postponed to Section \ref{Section6ProofMain2}. It
is based on constructing explicit couplings satisfying Condition \ref{NonExp}
and it relies on four steps. The first step is showing that the value
functions $v^{\delta}$ can be suitably approximated by using piecewise
constant open-loop policies. The second step is to interpret the system as a
stochastic differential equation (SDE) with respect to some random measure.
The third step is to embed the solutions of these SDE in a space of measures
satisfying a suitable linear constraint via the linear programming approach.
To conclude, the fourth step provides a constructive (pseudo-) coupling using
SDE estimates.

\section{Example of Application\label{Section4Examples}}

\subsection{Some Considerations on a Biological Model\label{Subsection4.1}}

We consider the model introduced in \cite{hasty_pradines_dolnik_collins_00} to
describe the regulation of gene expression. The model is derived from the
promoter region of bacteriophage $\lambda$. The simplification proposed by the
authors of \cite{hasty_pradines_dolnik_collins_00} consists in considering a
mutant system in which only two operator sites (known as OR2 and OR3) are
present. The gene cI expresses repressor (CI), which dimerizes and binds to
the DNA as a transcription factor in one of the two available sites. The site
OR2 leads to enhanced transcription, while OR3 represses transcription. Using
the notations in \cite{hasty_pradines_dolnik_collins_00}, we let $X_{1}$ stand
for the repressor, $X_{2}$ for the dimer, $D$ for the DNA promoter site,
$DX_{2}$ for the binding to the OR2 site, $DX_{2}^{\ast}$ for the binding to
the OR3 site and $DX_{2}X_{2}$ for the binding to both sites. We also denote
by $P$ the RNA\ polymerase concentration and by $n$ the number of proteins per
mRNA transcript. The dimerization, binding, transcription and degradation
reactions are summarized by%

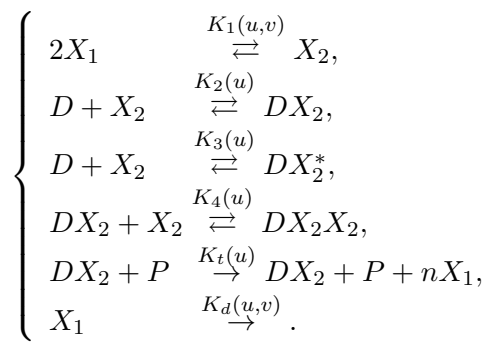
\[
\left\{
\begin{array}
[c]{l}%
2X_{1}\text{ \ \ \ \ \ \ \ \ \ }\overset{K_{1}\left(  u,v\right)
}{\rightleftarrows}X_{2},\\
D+X_{2}\text{ \ \ \ }\overset{K_{2}\left(  u\right)  }{\rightleftarrows}%
DX_{2},\\
D+X_{2}\text{ \ \ \ }\overset{K_{3}\left(  u\right)  }{\rightleftarrows}%
DX_{2}^{\ast},\\
DX_{2}+X_{2}\overset{K_{4}\left(  u\right)  }{\rightleftarrows}DX_{2}X_{2},\\
DX_{2}+P\text{ \ }\overset{K_{t}\left(  u\right)  }{\rightarrow}%
DX_{2}+P+nX_{1},\\
X_{1}\text{\ \ \ \ \ \ \ \ \ \ \ }\overset{K_{d}\left(  u,v\right)
}{\rightarrow}.
\end{array}
\right.
\]
The capital letters $K_{i},$ $1\leq i\leq4$ for the reversible reactions
correspond to couples of direct/reverse speed functions $k_{i},k_{-i},$ while
$K_{t}$ and $K_{d}$ only to direct speed functions $k_{t}$ and $k_{d}$. Host
DNA gyrase puts negative supercoils in the circular chromosome, causing
A-T-rich regions to unwind and drive transcription. This is why, in the model
written here, the binding speeds $k_{2}$ (to the promoter of the lysogenic
cycle $P_{RM}$) $,$ $k_{3}$ (to OR3), respectively $k_{4}$ and the reverse
speeds as well as the transcription speed $k_{t}$ are assumed to depend only
on a control on the host E-coli (denoted by $u).$ This control also acts on
the prophage and, hence, we find it, together with some additional control
$v,$ in the dimerization speed $k_{1}$ and the degradation speed $k_{d}$. 

\subsection{The Associated Mathematical Model\label{Subsection4.2} }

Let us briefly explain how a mathematical model can be associated to the
previous system.

\textbf{(A) Discrete/continuous components} 

We distinguish between components that are discrete (only affected by jumps)
and components that also have piecewise continuous dynamics. For the host
(E-Coli), we can have the following modes: either unoccupied DNA ($D=1$,
$DX_{2}=DX_{2}^{\ast}=DX_{2}X_{2}=0),$ or binding to OR2 ($D=DX_{2}^{\ast
}=DX_{2}X_{2}=0,$ $DX_{2}=1$), or to OR3 ($D=DX_{2}^{\ast}=DX_{2}X_{2}=0,$
$DX_{2}^{\ast}=1$) or to both sites ($D=DX_{2}^{\ast}=DX_{2}=0,$ $DX_{2}%
X_{2}=1$). It is obvious that these components are discrete and they belong to
$\mathbb{M}=\left\{  \gamma\in\left\{  0,1\right\}  ^{4}:\sum_{i=1}^{4}%
\gamma_{i}=1\right\}  .$ Every reaction involving at least one discrete
component will be of jump-type. We then have the four jump reactions %

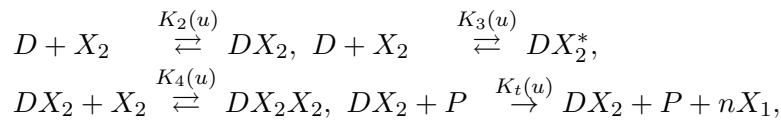
\[%
\begin{array}
[c]{l}%
D+X_{2}\text{ \ \ \ }\overset{K_{2}\left(  u\right)  }{\rightleftarrows}%
DX_{2},\text{ }D+X_{2}\text{ \ \ \ }\overset{K_{3}\left(  u\right)
}{\rightleftarrows}DX_{2}^{\ast},\\
DX_{2}+X_{2}\overset{K_{4}\left(  u\right)  }{\rightleftarrows}DX_{2}%
X_{2},\text{ }DX_{2}+P\text{ \ }\overset{K_{t}\left(  u\right)  }{\rightarrow
}DX_{2}+P+nX_{1},
\end{array}
\]

The couple repressor/dimer $\left(  X_{1},X_{2}\right)  $ has a different
scale and is averaged (has a deterministic evolution) between jumps$.$ Hence,
we deal with a hybrid model on $\mathbb{M\times%
\mathbb{R}
}^{2}.$ 

\textbf{(B) Jump mechanism}

Let us take an example. We assume that the current mode is unoccupied DNA
($\gamma=\left(  1,0,0,0\right)  $). The only jump reactions possible are
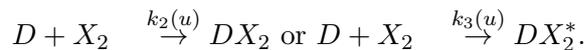
\[
D+X_{2}\ \ \ \overset{k_{2}\left(  u\right)  }{\rightarrow}DX_{2}\text{ or
}D+X_{2}\ \ \ \overset{k_{3}\left(  u\right)  }{\rightarrow}DX_{2}^{\ast}.
\]

The reaction $D+X_{2}$ \ \ \ $\overset{k_{2}\left(  u\right)  }{\rightarrow
}DX_{2}$ means that free DNA will be occupied by one dimer at OR2 position.
Therefore, we have a DNA and a dimer "consumed" and an OR2 binding "created".
The system jumps
\[
\text{from }\left(  1,0,0,0,x_{1},x_{2}\right)  \text{ to }\left(
0,1,0,0,x_{1},x_{2}-1\right)  .
\]
Of course, for a consistent mathematical model, since concentrations cannot be
negative, to $\left(  x_{1},x_{2}\right)  $ we actually add $\left(
0,-\min\left(  1,x_{2}\right)  \right)  $. 

The parameter $\lambda$ is chosen as the "propensity" function (i.e. the sum
of all possible reaction speeds)
\[
\lambda_{\left(  1,0,0,0\right)  }\left(  u,v\right)  =\lambda_{\left(
1,0,0,0\right)  }\left(  u\right)  =k_{2}\left(  u\right)  +k_{3}\left(
u\right)  ,
\]
The probability for the reaction $D+X_{2}\ \ \ \overset{k_{2}\left(  u\right)
}{\rightarrow}DX_{2}$ to take place is proportional to its reaction speed
(i.e. $\frac{k_{2}\left(  u\right)  }{k_{2}\left(  u\right)  +k_{3}\left(
u\right)  }).$

To summarize, one constructs%
\[
\left\{
\begin{array}
[c]{l}%
Q^{0}\left(  \left(  1,0,0,0\right)  ,\left(  u,v\right)  ,d\theta\right)  \\
=Q^{0}\left(  \left(  1,0,0,0\right)  ,u,d\theta\right)  =\frac{k_{2}\left(
u\right)  }{\lambda_{\left(  1,0,0,0\right)  }\left(  u\right)  }%
\delta_{\left(  0,1,0,0\right)  }\left(  d\theta\right)  +\frac{k_{3}\left(
u\right)  }{\lambda_{\left(  1,0,0,0\right)  }\left(  u\right)  }%
\delta_{\left(  0,0,1,0\right)  }\left(  d\theta\right)  ,\\
Q\left(  \left(  1,0,0,0\right)  ,\left(  x_{1},x_{2}\right)  ,u,d\theta
dy\right)  =\delta_{\left(  x_{1},x_{2}\right)  +g_{\gamma}\left(
\theta,\left(  x_{1},x_{2}\right)  ,u,v\right)  }\left(  dy\right)
Q^{0}\left(  \left(  1,0,0,0\right)  ,u,d\theta\right)  ,\text{ where}\\
g_{\left(  1,0,0,0\right)  }\left(  \theta,\left(  x_{1},x_{2}\right)
,u,v\right)  =\left(  0,-\min\left(  1,x_{2}\right)  \right)  1_{\theta
\in\left\{  \left(  0,1,0,0\right)  ,\left(  0,0,1,0\right)  \right\}  }.
\end{array}
\right.
\]

\begin{remark}
A special part is played by the transcription reaction
\[
DX_{2}+P\text{ \ }\overset{K_{t}\left(  u\right)  }{\rightarrow}%
DX_{2}+P+nX_{1}%
\]
which is a slow reaction. Details on a possible construction will be given in
Subsection \ref{Subsection4.3}. 
\end{remark}

\textbf{(C) Deterministic flow}

The deterministic behavior is governed by the (three) reactions%
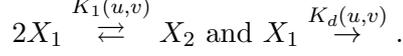
\[
2X_{1}\overset{K_{1}\left(  u,v\right)  }{\rightleftarrows}X_{2}\text{ and
}X_{1}\overset{K_{d}\left(  u,v\right)  }{\rightarrow}.
\]

For example, the reaction $2X_{1}\overset{k_{1}\left(  u,v\right)
}{\rightarrow}X_{2}$ needs two molecules of reactant $X_{1}$ and has
$k_{1}\left(  u,v\right)  $ as speed and produces one $X_{2}$. Then, following
the approach in \cite{hasty_pradines_dolnik_collins_00}, its contribution to
the vector field is proportional to the product of reactants at the power
molecules needed (i.e. $x_{1}^{2}$) and the speed $k_{1}\left(  u,v\right)  .$
We get $-2x_{1}^{2}k_{1}\left(  u,v\right)  $ for the reactant and
$+1x_{1}^{2}k_{1}\left(  u,v\right)  $ for the product. Adding the three
contributions, we get
\[
f_{\gamma}\left(  \left(  x_{1},x_{2}\right)  ,u,v\right)  =\left(
-2k_{1}\left(  u,v\right)  x_{1}^{2}+2k_{-1}\left(  u,v\right)  x_{2}%
-k_{d}\left(  u,v\right)  x_{1},k_{1}\left(  u,v\right)  x_{1}^{2}%
-k_{-1}\left(  u,v\right)  x_{2}\right)  .
\]
For further details on constructions related to systems of chemical reactions
in gene networks, the reader is referred to
\cite{crudu_debussche_radulescu_09},
\cite{crudu_Debussche_Muller_Radulescu_2012}, \cite{G8}, etc.

\subsection{A Toy Nondissipative Model\label{Subsection4.3}}

Let us exhibit a simple choice of coefficients in the study of phage
$\lambda.$ We consider $U=V=\left[  0,1\right]  $ (worst conditions, best
conditions for chemical reactions). For the reaction speed, we take%
\begin{align*}
k_{i}\left(  u\right)   &  =k_{i}\left(  u+u_{0}\right)  ,\text{ for }%
i\in\left\{  \pm2,\pm3,\pm4,t\right\}  \text{ for the jump reactions
determined by host,}\\
k_{1}\left(  u,v\right)   &  =\frac{1}{\alpha}uv,\text{ }k_{2}\left(
u,v\right)  =uv,k_{d}\left(  u,v\right)  =uv\text{ to reflect a certain
competition,}%
\end{align*}
for all $u,v\in\left[  0,1\right]  .$ Here, $u_{0}>0$ corresponds to the
slowest reaction speed, $k_{i}>0$ are real constants and $\alpha$ is some
maximal concentration level for repressor and dimer. We assume $P$ to toggle
between $0$ and $1$ and, as soon as $P$ toggles to $1$ a transcription burst
takes place. To take into account the slow aspect of the transcription
reaction (see \cite{hasty_pradines_dolnik_collins_00}), we actually take
\[
\mathbb{M}=\left\{  \left(  1,0,0,0,0\right)  ,\left(  0,1,0,0,0\right)
,\left(  0,1,0,0,1\right)  ,\left(  0,0,1,0,0\right)  ,\left(
0,0,0,1,0\right)  \right\}  .
\]
This means that binding to OR2 corresponds to two states $\left(
0,1,0,0,0\right)  $ (allowing transcription)$,\left(  0,1,0,0,1\right)  $
(when transcription has just taken place and is no longer allowed). We set%

\[%
\begin{array}
[c]{l}%
\lambda_{\gamma}\left(  u,v\right)  =\lambda_{\gamma}\left(  u\right)
=\left\{
\begin{array}
[c]{rcc}%
\left(  k_{2}+k_{3}\right)  \left(  u+u_{0}\right)  , & \text{if} &
\gamma=\left(  1,0,0,0,0\right)  ,\\
\left(  k_{-2}+k_{4}+k_{t}\right)  \left(  u+u_{0}\right)  , & \text{if} &
\gamma=\left(  0,1,0,0,0\right)  ,\\
\left(  k_{-2}+k_{4}\right)  \left(  u+u_{0}\right)  , & \text{if} &
\gamma=\left(  0,1,0,0,1\right)  ,\\
k_{-3}\left(  u+u_{0}\right)  , & \text{if} & \gamma=\left(  0,0,1,0,0\right)
,\\
k_{-4}\left(  u+u_{0}\right)  , & \text{if} & \gamma=\left(  0,0,0,1,0\right)
.
\end{array}
\right.  ,\\
Q^{0}\left(  \gamma,u\right)  =\\
\left\{
\begin{array}
[c]{rcc}%
\frac{k_{2}}{k_{2}+k_{3}}\delta_{\left(  0,1,0,0,0\right)  }+\frac{k_{3}%
}{k_{2}+k_{3}}\delta_{\left(  0,0,1,0,0\right)  }, & \text{if} &
\gamma=\left(  1,0,0,0,0\right)  ,\\
\frac{k_{-2}}{k_{-2}+k_{4}+k_{t}}\delta_{\left(  1,0,0,0,0\right)  }%
+\frac{k_{4}}{k_{-2}+k_{4}+k_{t}}\delta_{\left(  0,0,0,1,0\right)  }%
+\frac{k_{t}}{k_{-2}+k_{4}+k_{t}}\delta_{\left(  0,1,0,0,1\right)  }, &
\text{if} & \gamma=\left(  0,1,0,0,0\right)  ,\\
\frac{k_{-2}}{k_{-2}+k_{4}}\delta_{\left(  1,0,0,0,0\right)  }+\frac{k_{4}%
}{k_{-2}+k_{4}}\delta_{\left(  0,0,0,1,0\right)  }, & \text{if} &
\gamma=\left(  0,1,0,0,1\right)  ,\\
\delta_{\left(  1,0,0,0,0\right)  }, & \text{if} & \gamma=\left(
0,0,1,0,0\right)  ,\\
\delta_{\left(  0,1,0,0,0\right)  }, & \text{if} & \gamma=\left(
0,0,0,1,0\right)  .
\end{array}
\right. \\
g_{\gamma}\left(  \theta,\left(  x_{1},x_{2}\right)  \right)  =\left\{
\begin{array}
[c]{rcc}%
\left(  0,-\min\left(  1,x_{2}\right)  \right)  , & \text{if} & \gamma=\left(
1,0,0,0,0\right)  ,\\
\left(  0,\min\left(  1,\alpha-x_{2}\right)  \right)  , & \text{if} &
\gamma=\left(  0,1,0,0,\gamma_{5}\right)  ,\text{ }\gamma_{5}\in\left\{
0,1\right\}  ,\text{ }\theta=\left(  1,0,0,0,0\right)  ,\\
\left(  0,-\min\left(  1,x_{2}\right)  \right)  , & \text{if} & \gamma=\left(
0,1,0,0,\gamma_{5}\right)  ,\text{ }\gamma_{5}\in\left\{  0,1\right\}  ,\text{
}\theta=\left(  0,0,0,1,0\right)  ,\\
\left(  \min\left(  n,\alpha-x_{1}\right)  ,0\right)  , & \text{if} &
\gamma=\left(  0,1,0,0,0\right)  ,\text{ }\theta=\left(  0,1,0,0,1\right)  ,\\
\left(  0,\min\left(  1,\alpha-x_{2}\right)  \right)  , & \text{if} &
\gamma=\left(  0,0,1,0,0\right)  ,\text{ }\theta=\left(  1,0,0,0,0\right)  ,\\
\left(  0,\min\left(  1,\alpha-x_{2}\right)  \right)  , & \text{if} &
\gamma=\left(  0,0,0,1,0\right)  ,
\end{array}
\right. \\
f_{\gamma}\left(  \left(  x_{1},x_{2}\right)  ,u,v\right)  =\left(  f_{\gamma
}^{1}\left(  \left(  x_{1},x_{2}\right)  ,u,v\right)  ,f_{\gamma}^{2}\left(
\left(  x_{1},x_{2}\right)  ,u,v\right)  \right) \\
\text{ \ \ \ \ \ \ \ \ \ \ \ \ \ \ \ \ \ \ \ \ \ \ \ \ }=\left(  -\frac
{2uv}{\alpha}x_{1}^{2}+2uvx_{2}-uvx_{1},\frac{uv}{\alpha}x_{1}^{2}%
-uvx_{2}\right)  .
\end{array}
\]

The reader is invited to note that $\mathbb{M\times}\left[  0,\alpha\right]
^{2}$ is invariant with respect to the previous dynamics. One can either note
that for $x_{1},x_{2}\in\left[  0,\alpha\right]  $
\begin{align*}
f_{\gamma}^{1}\left(  \left(  0,x_{2}\right)  ,u,v\right)   &  =2x_{2}%
uv\geq0,\text{ }f_{\gamma}^{2}\left(  \left(  x_{1},0\right)  ,u,v\right)
=\frac{x_{1}^{2}}{\alpha}uv\geq0,\\
f_{\gamma}^{1}\left(  \left(  \alpha,x_{2}\right)  ,u,v\right)   &  =\left(
2x_{2}-3\alpha\right)  uv\leq0,\text{ }f_{\gamma}^{2}\left(  \left(
x_{1},\alpha\right)  ,u,v\right)  =\left(  \frac{x_{1}^{2}}{\alpha}%
-\alpha\right)  uv\leq0,
\end{align*}
or, alternatively, compute the normal cone to the frontier of $\left[
0,\alpha\right]  ^{2}$ (see \cite[Theorem 2.8]{G8}) to deduce that $\left[
0,\alpha\right]  ^{2}$ is invariant with respect to the deterministic
dynamics. Moreover, the definition of $g_{\gamma}$ guarantees that $\left(
X_{1},X_{2}\right)  $ does not leave $\left[  0,\alpha\right]  ^{2}.$ Let us
also note that, whenever $x=\left(  x_{1},x_{2}\right)  \in\left[
0,\alpha\right]  ^{2},$ $y=\left(  y_{1},y_{2}\right)  \in\left[
0,\alpha\right]  ^{2},$
\begin{align*}
&  \text{\ }\left\langle f_{\gamma}\left(  x,u,v\right)  -f_{\gamma}\left(
y,u,v\right)  ,x-y\right\rangle \\
&  =-uv\left(  \left[  \frac{2}{\alpha}\left(  x_{1}+y_{1}\right)  +1\right]
\left(  x_{1}-y_{1}\right)  ^{2}-\left[  2+\frac{1}{\alpha}\left(  x_{1}%
+y_{1}\right)  \right]  \left(  x_{1}-y_{1}\right)  \left(  x_{2}%
-y_{2}\right)  +\left(  x_{2}-y_{2}\right)  ^{2}\right)  \leq0.
\end{align*}
This is a simple consequence of the fact that
\[
\left[  2+\frac{1}{\alpha}\left(  x_{1}+y_{1}\right)  \right]  ^{2}-4\left[
\frac{2}{\alpha}\left(  x_{1}+y_{1}\right)  +1\right]  =\left[  2-\frac
{1}{\alpha}\left(  x_{1}+y_{1}\right)  \right]  ^{2}-4\leq0.
\]
We deduce that
\[
\sup_{v\in V}\inf_{w\in V}\left\langle f_{\gamma}\left(  x,u,v\right)
-f_{\gamma}\left(  y,u,w\right)  ,x-y\right\rangle \leq0\text{.}%
\]
Moreover, for all $\left(  x,y\right)  \in\left[  0,\alpha\right]  ^{2},$
\[
\sup_{v\in V}\inf_{w\in V}\left\langle f_{\gamma}\left(  x,0,v\right)
-f_{\gamma}\left(  y,0,w\right)  ,x-y\right\rangle =0\text{.}%
\]
It follows that one is not able to find any positive constant $c>0$ such that
\[
\sup_{v\in V}\inf_{w\in V}\left\langle f_{\gamma}\left(  x,0,v\right)
-f_{\gamma}\left(  y,0,w\right)  ,x-y\right\rangle \leq-c\left\vert
x-y\right\vert ^{2}\text{,}%
\]
for all $\left(  x,y\right)  \in\left[  0,\alpha\right]  ^{2}$ and we deal
with a non-dissipative system (unlike, for instance,
\cite{Benaim_LEBorgne_Malrieu_Zitt_2012}). 

Finally, the function $t\mapsto t-\min\left(  1,t\right)  ,$ $t\mapsto
t+\min\left(  k,\alpha-t\right)  $ are Lipschitz continuous with Lipschitz
constant $1$ on $\left[  0,\alpha\right]  $ for all $k>0.$ It follows that
$\underset{\theta\in\mathbb{M}}{\sup}\left\vert x+g_{\gamma}\left(
\theta,x\right)  -y-g_{\gamma}\left(  \theta,y\right)  \right\vert
\leq\left\vert x-y\right\vert .$

\section{Proof of the First Main Result (Theorem \ref{ThUniformConvMain}%
)\label{Section5ProofMain1}}

In order to prove Theorem \ref{ThUniformConvMain}, we proceed as follows.
First, we recall the link between the set of occupation measures $\Theta
_{0}^{\delta}\left(  \gamma,x\right)  $ and the family $\Theta^{\delta}\left(
\gamma,x\right)  $ (taken from \cite{G7}). Next, we prove that Condition
\ref{NonExp} implies the equicontinuity of the family of Abel-average value
functions $\left(  v^{\delta}\right)  _{\delta>0}.$ Finally, we recall the
results in \cite{GoreacSerea_TauberianPDMP_2014} on Abel-type theorems to conclude.

\subsection{Step 1 : Equicontinuity of Abel-average Values}

The following result corresponds to \cite[Theorem 7 and Corollary 8]{G7} for
this (less general) setting. It gives the link between the set of occupation
measures $\Theta_{0}^{\delta}\left(  \gamma,x\right)  $ and the family
$\Theta^{\delta}\left(  \gamma,x\right)  $.

\begin{theorem}
\label{ThLP}i) For every $x\in%
\mathbb{R}
^{N}$ and every $\delta>0,$
\[
v^{\delta}\left(  \gamma,x\right)  =\underset{\mu\in\Theta^{\delta}\left(
\gamma,x\right)  }{\inf}\int_{\mathbb{M\times}%
\mathbb{R}
^{N}\times\mathbb{U}}h\left(  \theta,y,u\right)  \mu\left(  d\theta
,dy,du\right)  .
\]

ii) For every $\left(  \gamma,x\right)  \in\mathbb{M\times}%
\mathbb{R}
^{N}$ and every $\delta>0,$ $\Theta^{\delta}\left(  \gamma,x\right)
=\overline{co}\left(  \Theta_{0}^{\delta}\left(  \gamma,x\right)  \right)  $.
\end{theorem}

We are now able to prove the following equicontinuity result. 

\begin{proposition}
\label{PropUniformContModuli}We assume Condition \ref{NonExp} to hold true.
Then, for every $\delta>0$ and every $\left(  \gamma,x,y\right)  \in
\mathbb{M}\times\mathbb{R}^{2N},$ one has
\[
\left\vert v^{\delta}\left(  \gamma,x\right)  -v^{\delta}\left(
\gamma,y\right)  \right\vert \leq Lip\left(  h\right)  \left\vert
x-y\right\vert .
\]

\end{proposition}

\begin{proof}
Let us fix $\delta>0$ and $\left(  \gamma,x,y\right)  \in\mathbb{M}%
\times\mathbb{R}^{2N}$. We only need to prove that, for every $\varepsilon
>0,$
\[
v^{\delta}\left(  \gamma,y\right)  \leq v^{\delta}\left(  \gamma,x\right)
+Lip\left(  h\right)  \left\vert x-y\right\vert +\varepsilon.
\]
By definition of $v^{\delta}$, for a fixed $\varepsilon>0,$ there exists some
$u\in\mathcal{A}_{ad}\left(  \mathbb{U}\right)  $ such that
\[
v^{\delta}\left(  \gamma,x\right)  +\frac{\varepsilon}{2}\geq\delta
\mathbb{E}\left[  \int_{0}^{\infty}e^{-\delta t}h\left(  \Gamma_{t}%
^{\gamma,x,u},X_{t}^{\gamma,x,u},u_{t}\right)  dt\right]  =\int%
_{\mathbb{M\times}\mathcal{%
\mathbb{R}
}^{N}\times\mathbb{U}}h\left(  \theta,z,w\right)  \mu_{\gamma,x,u}^{\delta
}\left(  d\theta,dz,dw\right)  .
\]
If $\mu$ is the coupling measure given by Condition \ref{NonExp} and
associated to $\frac{\varepsilon}{2},$ we deduce, using Theorem \ref{ThLP}.i
that
\begin{align*}
v^{\delta}\left(  \gamma,y\right)   &  \leq\int_{\left(  \mathbb{M\times
}\mathcal{%
\mathbb{R}
}^{N}\times\mathbb{U}\right)  ^{2}}h\left(  \theta^{\prime},z^{\prime
},w^{\prime}\right)  \mu\left(  d\theta,dz,dw,d\theta^{\prime},dz^{\prime
},dw^{\prime}\right)  \\
&  \leq\int_{\mathbb{M\times}\mathcal{%
\mathbb{R}
}^{N}\times\mathbb{U}}h\left(  \theta,z,w\right)  \mu_{\gamma,x,u}^{\delta
}\left(  d\theta,dz,dw\right)  +Lip\left(  h\right)  \left\vert x-y\right\vert
+\frac{\varepsilon}{2}\\
&  \leq v^{\delta}\left(  \gamma,x\right)  +Lip\left(  h\right)  \left\vert
x-y\right\vert +\varepsilon.
\end{align*}
The proof of our Proposition follows by recalling that $\varepsilon>0$ is arbitrary.
\end{proof}

\subsection{Step 2 : Proof of Theorem \ref{ThUniformConvMain}}

The previous results on existence of a continuity modulus uniform with respect
to the discount parameter $\delta>0$ allows us to prove the existence of a
limit value function as $\delta\rightarrow0.$  Before going to the proof of
Theorem \ref{ThUniformConvMain}, we recall the following.

\begin{lemma}
\label{EssentialLemma}(i) \textbf{\cite[Step 1 of Theorem 4.1]%
{GoreacSerea_TauberianPDMP_2014} }Let us assume that $\left(  v^{\delta
}\right)  _{\delta>0}$ is a relatively compact subset of \ $C\left(
\mathbb{M\times}%
\mathbb{R}
^{N};%
\mathbb{R}
\right)  $. Then, for every $v\in C\left(  \mathbb{M\times}%
\mathbb{R}
^{N};%
\mathbb{R}
\right)  $, every sequence $\left(  \delta_{m}\right)  _{m\geq1}$ such that
$\ \lim_{m\rightarrow\infty}\delta_{m}=0$ and $\left(  v^{\delta_{m}}\right)
_{m\geq1}$ converges uniformly to $v$ on $\mathbb{M\times}%
\mathbb{R}
^{N}$ and every $\varepsilon>0,$ there exists $T>0$ such that
\[
V_{t}\left(  \gamma,x\right)  \geq v\left(  \gamma,x\right)  -\varepsilon
,\text{ for all }\left(  \gamma,x\right)  \in\mathbb{M\times}%
\mathbb{R}
^{N}\text{ and all }t\geq T.
\]

(ii) \textbf{\cite[Theorem 4.1]{GoreacSerea_TauberianPDMP_2014} }Let us assume
that $\left(  v^{\delta}\right)  _{\delta>0}$ is a relatively compact subset
of \ $C\left(  \mathbb{M\times}%
\mathbb{R}
^{N};%
\mathbb{R}
\right)  $. Then, for every $v\in C\left(  \mathbb{M\times}%
\mathbb{R}
^{N};%
\mathbb{R}
\right)  $ and every sequence $\left(  \delta_{m}\right)  _{m\geq1}$ such that
$\ \lim_{m\rightarrow\infty}\delta_{m}=0$ and $\left(  v^{\delta_{m}}\right)
_{m\geq1}$ converges uniformly to $v$ on $\mathbb{M\times}%
\mathbb{R}
^{N},$ the following equality holds true
\[
\underset{t\rightarrow\infty}{\lim\inf}\sup_{\gamma\in\mathbb{M},\text{ }x\in%
\mathbb{R}
^{N}}\left\vert V_{t}\left(  \gamma,x\right)  -v\left(  \gamma,x\right)
\right\vert =0.
\]

(iii) \textbf{\cite[Remark 4.2]{GoreacSerea_TauberianPDMP_2014}} Let us assume
that $\left(  v^{\delta}\right)  _{\delta>0}$ converges uniformly to some
$v^{\ast}\in C\left(  \mathbb{M\times}%
\mathbb{R}
^{N};%
\mathbb{R}
\right)  $ as $\delta\rightarrow0$. Then the functions $\left(  V_{t}\right)
_{t>0}$ converge uniformly on $\mathbb{M\times}%
\mathbb{R}
^{N}$ to $v^{\ast}$%
\[
\underset{t\rightarrow\infty}{\lim}\sup_{\gamma\in\mathbb{M},\text{ }x\in%
\mathbb{R}
^{N}}\left\vert V_{t}\left(  \gamma,x\right)  -v^{\ast}\left(  \gamma
,x\right)  \right\vert =0.
\]

\end{lemma}

\begin{remark}
In the proof of \textbf{\cite[Theorem 4.1]{GoreacSerea_TauberianPDMP_2014},
}one gives the condition (i) in Step 1. However, in all generality, the
converse is only partial. Indeed, Step 2 (see \textbf{\cite[Page 174, Eq.
(10)]{GoreacSerea_TauberianPDMP_2014} }reads : For every $\varepsilon>0$,
there exists $m_{0}\geq1$ such that
\[
V_{\delta_{m}^{-1}}\left(  \gamma,x\right)  \leq v\left(  \gamma,x\right)
+\varepsilon,
\]
for all $m\geq m_{0}$ and all $\left(  \gamma,x\right)  \in\mathbb{M\times}%
\mathbb{R}
^{N}.$ Thus, the index $t$ of the subfamily $\left(  V_{t}\right)  _{t}$
converging to $v$ depends directly on the sequence $\left(  \delta_{m}\right)
_{m\geq1}.$ Of course, whenever the limit $v$ is independent on the choice of
$\delta$, so is $t$ and one gets (iii)$.$
\end{remark}

We are now able to complete the proof of Theorem \ref{ThUniformConvMain}.

\begin{proof}
[Proof of Theorem \ref{ThUniformConvMain}]Let us denote by
\begin{equation}
v^{\ast}\left(  \gamma,x\right)  :=\underset{\delta\rightarrow0}{\lim\sup
}\text{ }v^{\delta}\left(  \gamma,x\right)  ,\label{eq0}%
\end{equation}
for every $\left(  \gamma,x\right)  \in$ $\mathbb{M\times K}$ (the pointwise
$\lim\sup$). We fix, for the time being, some $\left(  \gamma,x\right)
\in\mathbb{M\times K}$. Then, there exists some sequence $\left(  \delta
_{m}\right)  _{m\geq1}$ such that $\ \underset{m\rightarrow\infty}{\lim}%
\delta_{m}=0$ and $\underset{m\rightarrow\infty}{\lim}v^{\delta_{m}}\left(
\gamma,x\right)  =v^{\ast}\left(  \gamma,x\right)  .$ Due to Condition
\ref{NonExp}, the sequence $\left(  v^{\delta_{m}}\right)  _{m\geq1}$ is
equicontinuous (see Proposition \ref{PropUniformContModuli}) and, by
definition, it is also bounded. Then, using Arzel\`{a}-Ascoli Theorem, it
follows that some subsequence (still denoted $\left(  v^{\delta_{m}}\right)
_{m\geq1}$) converges uniformly on $\mathbb{M\times K}$ to some limit function
$v\in C\left(  \mathbb{M\times K};%
\mathbb{R}
\right)  .$ In particular, $w\left(  \gamma,x\right)  =v^{\ast}\left(
\gamma,x\right)  .$ Using Lemma \ref{EssentialLemma} (i), one gets that
\begin{equation}
\underset{t\rightarrow\infty}{\lim\inf}V_{t}\left(  \gamma,x\right)  \geq
v^{\ast}\left(  \gamma,x\right)  .\label{eq1}%
\end{equation}
Obviously, this argument can be repeated for every $\left(  \gamma,x\right)
\in\mathbb{M\times K}$. Let us now consider $w\in C\left(  \mathbb{M\times K};%
\mathbb{R}
\right)  $ to be an adherence point of the relatively compact family $\left(
v^{\delta}\right)  _{\delta>0}$. Then, using\ Lemma \ref{EssentialLemma} (ii),
one establishes the existence of some increasing sequence $\left(
t_{n}\right)  _{n\geq1}$ such that $\underset{n\rightarrow\infty}{\lim}%
t_{n}=\infty$ and
\[
\lim_{n\rightarrow\infty}\sup_{\gamma\in\mathbb{M},x\in%
\mathbb{R}
^{N}}\left\vert V_{t_{n}}\left(  \gamma,x\right)  -w\left(  \gamma,x\right)
\right\vert =0.
\]
In particular, it follows that
\begin{equation}
w\left(  \gamma,x\right)  \geq\underset{t\rightarrow\infty}{\lim\inf}%
V_{t}\left(  \gamma,x\right)  ,\label{eq2}%
\end{equation}
for all $\left(  \gamma,x\right)  \in\mathbb{M\times K}$. Combining
(\ref{eq0}),(\ref{eq1}) and (\ref{eq2}), one deduces that the unique adherence
point of $\left(  v^{\delta}\right)  _{\delta>0}$ is $v^{\ast}$. The
convergence of $\left(  V_{t}\right)  _{t>0}$ follows by invoking Lemma
\ref{EssentialLemma} (iii). Our Theorem is now complete.
\end{proof}

\section{Proof of the Second Main Result (Theorem \ref{ThMainFirst}%
)\label{Section6ProofMain2}}

The proof of Theorem \ref{ThMainFirst} is constructive and relies on four
steps. We begin with recalling the Hamilton-Jacobi integrodifferential systems
satisfied by the Abel-average functions and Krylov's shaking the coefficient
method. The first step is showing that the value functions $v^{\delta}$ can be
suitably approximated by using piecewise constant open-loop policies. The
proof in this part strongly rely on the tools presented before. The second
step is to interpret the system as a stochastic differential equation (SDE)
with respect to some random measure. The third step is to embed the solutions
of these SDE in a space of measures satisfying a suitable linear constraint
via the linear programming approach. To conclude, the fourth step provides a
constructive (pseudo-) coupling using SDE estimates.

\subsection{Krylov's Shaking the Coefficients\label{Subsection6.1Krylov}}

For every $\delta>0,$ the value function $v^{\delta}$ is known to be the
unique bounded, uniformly continuous viscosity solution of the Hamilton-Jacobi
integro-differential system%
\begin{equation}
\delta v^{\delta}\left(  \gamma,x\right)  +H\left(  \gamma,x,\partial
_{x}v^{\delta}\left(  \gamma,x\right)  ,v^{\delta}\right)  =0, \label{HJB}%
\end{equation}
where the Hamiltonian is defined by setting%
\begin{align*}
&  H\left(  \gamma,x,p,\varphi\right) \\
&  :=\underset{u\in\mathbb{U}}{\sup}\left[  -h\left(  \gamma,x,u\right)
-\left\langle f_{\gamma}\left(  x,u\right)  ,p\right\rangle -\lambda_{\gamma
}\left(  x,u\right)  \int_{\mathbb{M}}\left(  \varphi\left(  \theta
,x+g_{\gamma}\left(  \theta,x,u\right)  \right)  -\varphi\left(
\gamma,x\right)  \right)  Q^{0}\left(  \gamma,u,d\theta\right)  \right]  ,
\end{align*}
for all $x,p\in%
\mathbb{R}
^{N}$ and all bounded function $\varphi:\mathbb{M}\times%
\mathbb{R}
^{N}$ $\longrightarrow%
\mathbb{R}
.$ For further details on the subject, the reader is referred to
\cite{Soner86_2}.

Although uniformly continuous, the value functions $v^{\delta}$ are, in
general, not of class $C_{b}^{1}.$ However, adapting the method introduced in
\cite{krylov_00} (see also \cite{barles_jakobsen_02}), $v^{\delta}$ can be
seen as the supremum over regular subsolutions of the system (\ref{HJB}).
Alternatively, one can give a variational formulation of $v^{\delta}$ with
respect to an explicit set of constraints. We recall the following basic
elements taken from \cite{G7}.

We begin by perturbing the coefficients and consider an extended
characteristic triple

- $\overline{f}_{\gamma}:%
\mathbb{R}
^{N}\times\mathbb{U\times}\overline{B}\left(  0,1\right)  \longrightarrow%
\mathbb{R}
^{N},$ $f_{\gamma}\left(  x,u^{1},u^{2}\right)  =f_{\gamma}\left(
x+u^{2},u^{1}\right)  ,$ $u_{1}\in\mathbb{U},u_{2}\in\overline{B}\left(
0,1\right)  ,$ $\gamma\in\mathbb{M},$

- $\overline{\lambda}_{\gamma}:%
\mathbb{R}
^{N}\times\mathbb{U\times}\overline{B}\left(  0,1\right)  \longrightarrow%
\mathbb{R}
^{N},$ $\lambda_{\gamma}\left(  x,u^{1},u^{2}\right)  =\lambda_{\gamma}\left(
x+u^{2},u^{1}\right)  ,$ $u_{1}\in\mathbb{U},u_{2}\in\overline{B}\left(
0,1\right)  ,$ $\gamma\in\mathbb{M},$

- $\overline{Q}:%
\mathbb{R}
^{N}\times\mathbb{U\times}\overline{B}\left(  0,1\right)  \longrightarrow
\mathcal{P}\left(
\mathbb{R}
^{N}\right)  ,$ $\overline{Q}\left(  \gamma,x,u^{1},u^{2},A\right)  =Q\left(
\gamma,x+u^{2},u^{1},A+\left(  0,u^{2}\right)  \right)  ,$ where $A+\left(
0,u^{2}\right)  =\left\{  \left(  a_{1},a_{2}+u^{2}\right)  :\left(
a_{1},a_{2}\right)  \in A\right\}  ,$ for all $x\in%
\mathbb{R}
^{N}$, $u^{1}\in U$, $u^{2}\in\overline{B}\left(  0,1\right)  $ and all Borel
set $A\subset\mathbb{M\times}%
\mathbb{R}
^{N}.$

One can easily construct the process $\left(  \Gamma^{\gamma,x,u^{1},u^{2}%
},X^{\gamma,x,u^{1},u^{2}}\right)  $ with $u=\left(  u^{1},u^{2}\right)
\in\mathcal{A}_{ad}\left(  \mathbb{U\times}\overline{B}\left(  0,1\right)
\right)  $. The initial process associated to $\left(  f,\lambda,Q\right)  $
can be obtained by imposing $u^{2}=0.$ Let us note that, with this
construction,
\[
\overline{Q}\left(  \gamma,x,u^{1},u^{2},d\theta dy\right)  =\delta
_{x+g_{\gamma}\left(  \theta,x+u^{2},u^{1}\right)  }\left(  dy\right)
Q^{0}\left(  \gamma,u^{1},d\theta\right)  .
\]

\subsection{Step 1: Piecewise Constant Open-loop
Policies\label{Subsection6.2PiecewiseCstePolicies}}

The aim of this subsection is to show that the value functions $v^{\delta}$
can be approximated by functions in which the control processes are piecewise
(in time) constant. For Brownian diffusions, this type of result has been
proven in \cite{Krylov_step_99}. In this section we adapt the method of
\cite{Krylov_step_99} to our setting by hinting to the modifications whenever
necessary. Following \cite{Krylov_step_99}, for all $n\geq1,$ we introduce the
value function%
\[
v^{\delta,n}\left(  \gamma,x\right)  =\inf_{u\in\mathcal{A}_{ad}^{n}}%
\delta\mathbb{E}\left[  \int_{0}^{\infty}e^{-\delta t}h\left(  \Gamma
_{t}^{\gamma,x,u,0},X_{t}^{\gamma,x,u,0},u_{t}\right)  dt\right]  ,
\]
for all $\left(  \gamma,x\right)  \in\mathbb{M\times}%
\mathbb{R}
^{N}.$

The main result of the subsection is the following.

\begin{theorem}
\label{ThPiecewiseCtControl}Let us assume that there exists a compact, convex
set $\mathbb{K}$ which is invariant with respect to the controlled PDMP with
characteristics $\left(  f,\lambda,Q\right)  $. Then, for every $\delta>0,$
the value functions $v^{\delta,n}$ converge uniformly to $v^{\delta}$ as the
discretization step $n$ increases to infinity
\[
\lim_{n\rightarrow\infty}\sup_{\gamma\in\mathbb{M}\text{, }x\in\mathbb{K}%
}\left\vert v^{\delta}\left(  \gamma,x\right)  -v^{\delta,n}\left(
\gamma,x\right)  \right\vert =0.
\]

\end{theorem}

The proof relies on the same arguments as those developed in
\cite{Krylov_step_99} combined with dynamic programming principles. Let us
briefly explain the approach. For every $n\geq1$, one begins by proving a
dynamic programming principle for $v^{\delta,n}$ and involving $T\wedge T_{1}$
as intermediate time, for $T\in n^{-1}\mathbb{%
\mathbb{N}
}$. The arguments are essentially the same as those in \cite{Soner86_2} and we
only specify when the structure of $\mathcal{A}_{ad}^{n}$ intervenes. Next,
one takes a sequence of smooth functions $\left(  v_{\left(  \varepsilon
\right)  }^{\delta,n}\right)  _{\varepsilon>0}$ converging uniformly to
$v^{\delta,n}$ by adapting to the present framework Krylov's shaking of
coefficients method introduced in \cite{krylov_00} (see also
\cite{barles_jakobsen_02} or \cite{G8} for the PDMP case). Then, one proceeds
by writing the Hamilton-Jacobi integrodifferential system satisfied by
$v_{\left(  \varepsilon\right)  }^{\delta,n}$. This equation is $\varepsilon
-$close to the one satisfied by $v^{\delta}$ (with a uniform behavior w.r.t.
$n\geq1$)$.$ Our assertion follows by integrating this subsolution condition
with respect to the law of the piecewise deterministic Markov process then
allowing $\varepsilon\rightarrow0$. For our reader's convenience, we have
indicated the main modifications and arguments in the Appendix.

\begin{remark}
\label{RemarkNonExpCtCtrl}If the invariance condition holds true, then, by
applying this result, one only needs to check that the nonexpansive Condition
\ref{NonExp} holds true for all $u\in\mathcal{A}_{ad}^{n}\left(
\mathbb{U}\right)  $ for all $n$ large enough (larger than some
$n_{\varepsilon}$)$.$
\end{remark}

\subsection{Step 2 : Associated Random Measures and Stochastic Differential
Equations\label{Subsection6.3RandomMeasures}}

Let us fix $\gamma_{0}\in\mathbb{M},$ $x_{0}\in\mathbb{R}^{N}$ and $\left(
u,v\right)  \in\mathcal{A}_{ad}\left(  U\times V\right)  .$ The following
construction is quite standard and makes the object of \cite[Section
26]{davis_93} for more general PDMP (without control) and \cite[Section
41]{davis_93} (when control is present).\ We let $S_{0}=T_{0}=0$, $S_{n}%
=T_{n}-T_{n-1},$ $\ $for all $n\geq1$ and $\xi_{n}=\left(  S_{n},\gamma
_{T_{n}}^{\gamma_{0},x_{0},u,v},X_{T_{n}}^{\gamma_{0},x_{0},u,v}\right)  .$ We
look at the process $\left(  \gamma,X\right)  $ under $\mathbb{P}^{\gamma
_{0},x_{0},u,v}$ (which depends on both the initial state $\left(  \gamma
_{0},x_{0}\right)  $ and the control couple $\left(  u,v\right)  $, but,
having fixed these elements and for notation purposes, this dependency will be
dropped)$.$ By abuse of notation, we let
\[
u_{s}:=u_{1}\left(  \gamma_{0},x_{0},s\right)  1_{0\leq s\leq T_{1}}%
+\sum_{n\geq1}u_{n+1}\left(  \gamma_{T_{n}}^{\gamma_{0},x_{0},u,v},X_{T_{n}%
}^{\gamma_{0},x_{0},u,v},s-T_{n}\right)  1_{T_{n}<s\leq T_{n+1}},
\]
(and similar for $v$). We denote by $\mathbb{F}$ the filtration $\left(
\mathcal{F}_{\left[  0,t\right]  }:=\sigma\left\{  \left(  \gamma_{r}%
^{\gamma_{0},x_{0},u,v},X_{r}^{\gamma_{0},x_{0},u,v}\right)  :r\in\left[
0,t\right]  \right\}  \right)  _{t\geq0}.$ The predictable $\sigma$-algebra
will be denoted by $\mathcal{P}$ and the progressive $\sigma$-algebra by
$Prog.$ For the general structure of predictable processes, the reader is
referred to \cite[Section 26]{davis_93}, \cite[Proposition 4.2.1]{Jacobsen} or
\cite[Appendix A2, Theorem T34]{Bremaud_1981}. In particular, due to the
previous notations, it follows that $u$ and $v$ are predictable.

As usual, we introduce the random measure $\overline{p}$ on $\Omega
\times\left(  0,\infty\right)  \times\mathbb{M\times%
\mathbb{R}
}^{N}$ by setting%
\[
\overline{p}\left(  \omega,A\right)  =\sum_{k\geq1}1_{\left(  T_{k}\left(
\omega\right)  ,\left(  \gamma_{T_{k}}^{\gamma_{0},x_{0},u,v},X_{T_{k}%
}^{\gamma_{0},x_{0},u,v}\right)  \left(  \omega\right)  \right)  \in A},\text{
for all }\omega\in\Omega,\text{ }A\in\mathcal{B}\left(  0,\infty\right)
\times\mathcal{B}\left(  \mathbb{M\times%
\mathbb{R}
}^{N}\right)  .
\]
The compensator of $\overline{p}$ is
\[
\widehat{\overline{p}}\left(  dsdyd\theta\right)  =\lambda\left(  \gamma
_{s-}^{\gamma_{0},x_{0},u,v},u_{s}\right)  \delta_{X_{s-}^{\gamma_{0}%
,x_{0},u,v}+g_{\gamma_{s-}^{\gamma_{0},x_{0},u,v}}\left(  \theta
,X_{s-}^{\gamma_{0},x_{0},u,v},u_{s},v_{s}\right)  }\left(  dy\right)
Q^{0}\left(  \gamma_{s-}^{\gamma_{0},x_{0},u,v},u_{s},d\theta\right)  ds.
\]
and the compensated martingale measure (see \cite[Proposition 26.7]{davis_93})
is given by $\overline{q}:=\overline{p}-\widehat{\overline{p}}.$

By construction, for our model, on $\left[  T_{n-1},T_{n}\right)  ,$
$X_{t}^{\gamma_{0},x_{0},u,v}$ is a deterministic function of $X_{T_{n-1}%
}^{\gamma_{0},x_{0},u,v},$ $\gamma_{T_{n-1}}^{\gamma_{0},x_{0},u,v},$
$u_{n}\left(  X_{T_{n-1}}^{\gamma_{0},x_{0},u,v},\gamma_{T_{n-1}}^{\gamma
_{0},x_{0},u,v},\cdot-T_{n-1}\right)  $ and $v_{n}\left(  X_{T_{n-1}}%
^{\gamma_{0},x_{0},u,v},\gamma_{T_{n-1}}^{\gamma_{0},x_{0},u,v},\cdot
-T_{n-1}\right)  $. In this particular framework,
\begin{align*}
&  X_{T_{n}}^{\gamma_{0},x_{0},u,v}=\Phi_{T_{n}-T_{n-1}}^{0,X_{T_{n-1}%
}^{\gamma_{0},x_{0},u,v},u_{n}\left(  \gamma_{T_{n-1}}^{\gamma_{0},x_{0}%
,u,v},X_{T_{n-1}}^{\gamma_{0},x_{0},u,v},\cdot\right)  ,v_{n}\left(
\gamma_{T_{n-1}}^{\gamma_{0},x_{0},u,v},X_{T_{n-1}}^{\gamma_{0},x_{0}%
,u,v},\cdot\right)  ;\gamma_{T_{n-1}}^{\gamma_{0},x_{0},u,v}}\\
&  \text{ \ \ \ }+g_{\gamma_{T_{n-1}}^{\gamma_{0},x_{0},u,v}}\left(
\gamma_{T_{n}}^{\gamma_{0},x_{0},u,v},u_{n}\left(  \gamma_{T_{n-1}}%
^{\gamma_{0},x_{0},u,v},X_{T_{n-1}}^{\gamma_{0},x_{0},u,v},T_{n}%
-T_{n-1}\right)  ,v_{n}\left(  \gamma_{T_{n-1}}^{\gamma_{0},x_{0}%
,u,v},X_{T_{n-1}}^{\gamma_{0},x_{0},u,v},T_{n}-T_{n-1}\right)  \right)  ,
\end{align*}
hence being a a deterministic function of $S_{n},X_{T_{n-1}}^{\gamma_{0}%
,x_{0},u,v},\gamma_{T_{n-1}}^{\gamma_{0},x_{0},u,v}.$ It follows that the
filtration $\mathbb{F}$ is actually generated by the marked point process
$\left(  T_{k},\gamma_{T_{k}}^{\gamma_{0},x_{0},u,v}\right)  _{k\geq0}$. As a
consequence, $v_{n+1}=v_{n+1}\left(  X_{T_{n}}^{\gamma_{0},x_{0},u,v}%
,\gamma_{T_{n}}^{\gamma_{0},x_{0},u,v},\cdot\right)  $ is a deterministic
function of $T_{1},...,T_{n},\gamma_{T_{1}}^{\gamma_{0},x_{0},u,v}%
,...,\gamma_{T_{n}}^{\gamma_{0},x_{0},u,v}$ still denoted by $v_{n+1}\left(
T_{1},...,T_{n},\gamma_{T_{1}}^{\gamma_{0},x_{0},u,v},...,\gamma_{T_{n}%
}^{\gamma_{0},x_{0},u,v},\cdot\right)  .$ In the case when $\left(
u,v\right)  \in\mathcal{A}_{ad}^{m}\left(  U\times V\right)  $ for some
$m\geq1$ are piecewise constant, $v_{n+1}$ is of type
\[
\sum_{k\geq0}v_{n+1}^{k}\left(  T_{1},...,T_{n},\gamma_{T_{1}}^{\gamma
_{0},x_{0},u,v},...,\gamma_{T_{n}}^{\gamma_{0},x_{0},u,v}\right)  1_{\left(
\frac{k}{m},\frac{k+1}{m}\right]  }\left(  t\right)  .
\]
Similar assertion hold true for $u.$

We now define the random measure $p$ on $\Omega\times\left(  0,\infty\right)
\times\mathbb{M}$ by setting%
\[
p\left(  \omega,A\right)  =\overline{p}\left(  \omega,A\times\mathbb{%
\mathbb{R}
}^{N}\right)  ,\text{ for all }\omega\in\Omega,\text{ }A\in\mathcal{B}\left(
0,\infty\right)  \times\mathcal{B}\left(  \mathbb{M}\right)  .
\]
The properties of $\overline{p}$ imply that the compensator of $p$ is
\[
\widehat{p}\left(  dsd\theta\right)  =\lambda\left(  \gamma_{s-}^{\gamma
_{0},x_{0},u,v},u_{s}\right)  Q^{0}\left(  \gamma_{s-}^{\gamma_{0},x_{0}%
,u,v},u_{s},d\theta\right)  ds
\]
and
\[
q\left(  dsd\theta\right)  =p\left(  dsd\theta\right)  -\lambda\left(
\gamma_{s-}^{\gamma_{0},x_{0},u,v},u_{s}\right)  Q^{0}\left(  \gamma
_{s-}^{\gamma_{0},x_{0},u,v},u_{s},d\theta\right)  ds
\]
is its martingale measure. Following the general theory of integration with
respect to random measures (see, for example \cite{Ikeda_Watanabe_1981}), the
second state component can be identified with the unique solution of the
stochastic differential equation (SDE)%
\[
\left\{
\begin{array}
[c]{l}%
dX_{t}^{\gamma_{0},x_{0},u,v}=f_{\gamma_{t}^{\gamma_{0},x_{0},u,v}}\left(
X_{t}^{\gamma_{0},x_{0},u,v},u_{t},v_{t}\right)  dt+\int_{\mathbb{M}}%
g_{\gamma_{t-}^{\gamma_{0},x_{0},u,v}}\left(  \theta,X_{t-}^{\gamma_{0}%
,x_{0},u,v},u_{t},v_{t}\right)  p\left(  dtd\theta\right)  ,t\geq0,\\
X_{t}^{\gamma_{0},x_{0},u,v}=x_{0},\text{ }\mathbb{P}-a.s.
\end{array}
\right.
\]

\subsection{Step 3 : Measure Embedding of
Solutions\label{Subsection6.4MeasureEmbedding}}

More general, whenever $w$ is an $\mathbb{F}$-predictable process, we can
consider the equation%
\[
\left\{
\begin{array}
[c]{l}%
dY_{t}^{y_{0},w}=f_{\gamma_{t}^{\gamma_{0},x_{0},u,v}}\left(  Y_{t}^{y_{0}%
,w},u_{t},w_{t}\right)  dt+\int_{\mathbb{M}}g_{\gamma_{t-}^{\gamma_{0}%
,x_{0},u,v}}\left(  \theta,Y_{t-}^{y_{0},w},u_{t},w_{t}\right)  p\left(
dtd\theta\right)  ,t\geq0,\\
Y_{0}^{y_{0},w}=y_{0},\text{ }\mathbb{P}-a.s.
\end{array}
\right.
\]
The assumptions on the coefficients $f$ and $g$ guarantee that, for every
$y_{0}\in%
\mathbb{R}
^{N}$ and every predictable, $V-$valued process $w,$ this equation admits a
unique solution $Y^{y_{0},w}$. We fix $\delta>0$ and consider some (arbitrary)
regular test function $\phi\in C_{b}^{1}\left(  \mathbb{M\times}%
\mathbb{R}
^{N};%
\mathbb{R}
\right)  .$ It\^{o}'s formula (see \cite[Chapter II, Theorem 5.1]%
{Ikeda_Watanabe_1981}) applied to $\delta e^{-\delta\cdot}\phi\left(
\gamma_{\cdot}^{\gamma_{0},x_{0},u,v},Y_{\cdot}^{y_{0},w}\right)  $ on
$\left[  0,T\right]  $ yields%
\begin{align*}
&  \delta e^{-\delta T}\mathbb{E}\left[  \phi\left(  \gamma_{T}^{\gamma
_{0},x_{0},u,v},Y_{T}^{y_{0},w}\right)  \right] \\
&  =\delta\phi\left(  \gamma_{0},y_{0}\right) \\
&  +\mathbb{E}\left[  \int_{0}^{T}\delta e^{-\delta t}\left(
\begin{array}
[c]{c}%
-\delta\phi\left(  \gamma_{t},Y_{t}\right)  +\left\langle f_{\gamma_{t}%
}\left(  Y_{t},u_{t},w_{t}\right)  ,\partial_{x}\phi\left(  \gamma_{t}%
,Y_{t}\right)  \right\rangle \\
+\lambda\left(  \gamma_{t},u_{t}\right)  \int_{\mathbb{M}}\left(  \phi\left(
\theta,Y_{t}+g_{\gamma_{t}}\left(  \theta,Y_{t},u_{t},w_{t}\right)  \right)
-\phi\left(  \gamma_{t},Y_{t}\right)  \right)  Q^{0}\left(  \gamma_{t}%
,u_{t},d\theta\right)
\end{array}
\right)  dt\right]  ,
\end{align*}
where we have denoted by $\left(  \gamma_{t},Y_{t}\right)  =\left(  \gamma
_{t}^{\gamma_{0},x_{0},u,v},Y_{t}^{y_{0},w}\right)  .$ By letting
$T\rightarrow\infty$, it follows that the occupation measure $\mu^{y_{0},w}%
\in\mathcal{P}\left(  \mathbb{M\times}%
\mathbb{R}
^{N}\times U\times V\right)  $ given by
\[
\mu^{y_{0},w}\left(  A\right)  =\mathbb{E}\left[  \int_{0}^{\infty}\delta
e^{-\delta t}1_{A}\left(  \gamma_{t}^{\gamma_{0},x_{0},u,v},Y_{t}^{y_{0}%
,w},u_{t},w_{t}\right)  dt\right]  ,\text{ for }A\in\mathcal{B}\left(
\mathbb{M\times}%
\mathbb{R}
^{N}\times U\times V\right)
\]
satisfies%
\[
\int_{\mathbb{M\times}\mathcal{%
\mathbb{R}
}^{N}\times\mathbb{U}}\left(  \mathcal{L}^{u,v}\phi\left(  \theta,y\right)
+\delta\left(  \phi(\gamma,x)-\phi\left(  \theta,y\right)  \right)  \right)
\mu\left(  d\theta,dy,du,dv\right)  =0.
\]
We recall that $\mathcal{L}^{u,v}$ is the generator given by (\ref{Luv}).

There is no reason for the couple $\left(  \gamma^{\gamma_{0},x_{0}%
,u,v},Y^{y_{0},w}\right)  $ to be associated to a $U\times V$-valued piecewise
open-loop control couple. Nevertheless, the previous arguments show that the
occupation measure $\mu^{y_{0},w}$ belongs to $\Theta^{\delta}\left(
\gamma_{0},y_{0}\right)  $ (see (\ref{Thetadelta})).

\begin{remark}
\label{RemInvarianceY}Let us note that if there exists a set $\mathbb{K}$
invariant with respect to the PDMP driven by $\left(  f,\lambda,Q\right)  ,$
then, for all $\gamma_{0}\in\mathbb{M},$ $y_{0}\in\mathbb{K}$, the occupation
measures $\mu\in\Theta_{0}^{\delta}\left(  \gamma_{0},y_{0}\right)  $ satisfy
the support condition $\mu\left(  \mathbb{M}\times\mathbb{K\times}U\times
V\right)  $ $=1$. Then, by Theorem \ref{ThLP}, the same holds true for
$\Theta^{\delta}\left(  \gamma_{0},y_{0}\right)  $ and, hence, $Y^{y_{0},w}$
takes its values in $\mathbb{K}$. Alternatively, one can use \cite[Theorem 2.8
(ii)]{G8}.
\end{remark}

\subsection{Step 4 : Coupling via the Random
Measure\label{Subsection6.5Coupling}}

As in the previous arguments, one can define a measure $\mu\in\mathcal{P}%
\left(  \left(  \mathbb{M}\times\mathbb{R}^{N}\times U\times V\right)
^{2}\right)  $ by setting%
\[
\mu\left(  A\times B\right)  =\mathbb{E}\left[  \int_{0}^{\infty}\delta
e^{-\delta t}1_{A}\left(  \gamma_{t}^{\gamma_{0},x_{0},u,v},X_{t}^{\gamma
_{0},x_{0},u,v},u_{t},v_{t}\right)  1_{B}\left(  \gamma_{t}^{\gamma_{0}%
,x_{0},u,v},Y_{t}^{y_{0},w},u_{t},w_{t}\right)  dt\right]  ,
\]
whenever $A\in\mathcal{B}\left(  \left(  \mathbb{M}\times\mathbb{R}^{N}\times
U\times V\right)  ^{2}\right)  .$ It is clear that
\begin{align*}
&  \underset{\left(  \mathbb{M}\times%
\mathbb{R}
^{N}\times\mathbb{U}\right)  ^{2}}{\int}\left\vert h\left(  \theta,z,w\right)
-h\left(  \theta^{\prime},z^{\prime},w^{\prime}\right)  \right\vert \mu\left(
d\theta,dz,dw,d\theta^{\prime},dz^{\prime},dw^{\prime}\right)  \\
&  =\mathbb{E}\left[  \int_{0}^{\infty}\delta e^{-\delta t}\left\vert h\left(
\gamma_{t}^{\gamma_{0},x_{0},u,v},X_{t}^{\gamma_{0},x_{0},u,v},u_{t}%
,v_{t}\right)  -h\left(  \gamma_{t}^{\gamma_{0},x_{0},u,v},Y_{t}^{y_{0}%
,w},u_{t},w_{t}\right)  \right\vert dt\right]  ,
\end{align*}
$\mu\left(  A\times\left(  \mathbb{M}\times\mathbb{R}^{N}\times U\times
V\right)  \right)  =\mu_{\gamma_{0},x_{0,}u,v}^{\delta}\in\Theta_{0}^{\delta
}\left(  \gamma_{0},x_{0,}\right)  $ and $\mu\left(  \left(  \mathbb{M}%
\times\mathbb{R}^{N}\times U\times V\right)  \times B\right)  =\mu^{y_{0}%
,w}\in\Theta^{\delta}\left(  \gamma_{0},y_{0}\right)  ,$ where $\mu^{y_{0},w}$
given in the previous arguments. Convenient estimates for this integral term
imply the condition (\ref{NonExp}) and, hence, the results on existence of a
limit value function. In fact (see Remark \ref{RemarkNonExpCtCtrl}), in order
to prove Theorem \ref{ThMainFirst}, it suffices to provide good estimates when
the process is constructed with piecewise constant (in time) policies $\left(
u,v\right)  \in\mathcal{A}_{ad}^{n}\left(  U\times V\right)  .$ This is done
by the following.

\begin{lemma}
We assume Condition \ref{NonExpCondition} to hold true. Moreover, we assume
that there exists a compact set $\mathbb{K}$ invariant with respect to the
PDMP governed by $\left(  f,\lambda,Q\right)  $. Then, there exists $\omega:%
\mathbb{R}
_{+}\longrightarrow%
\mathbb{R}
_{+}$ such that $lim_{\varepsilon\rightarrow0}\omega\left(  \varepsilon
\right)  =0$ and, for every $n\geq1$ and every $\left(  u,v\right)
\in\mathcal{A}_{ad}^{n}\left(  U\times V\right)  ,$ there exists $w$
predictable with respect to the filtration $\mathbb{F}^{\gamma_{0},x_{0},u,v}$
such that
\[
\mathbb{E}\left[  \int_{0}^{\infty}\delta e^{-\delta t}\left\vert h\left(
\gamma_{t}^{\gamma_{0},x_{0},u,v},X_{t}^{\gamma_{0},x_{0},u,v},u_{t}%
,v_{t}\right)  -h\left(  \gamma_{t}^{\gamma_{0},x_{0},u,v},Y_{t}^{y_{0}%
,w},u_{t},w_{t}\right)  \right\vert dt\right]  \leq\omega\left(
n^{-1}\right)  .
\]

\end{lemma}

\begin{proof}
\underline{Step 0.}

Let us define a set-valued function%
\begin{align*}
&  \mathbb{M}\times\mathbb{K}^{2}\times U\times V\ni\left(  \gamma
,x,y,u,v\right)  \rightsquigarrow\Xi\left(  \gamma,x,y,u,v\right) \\
&  :=\left\{
\begin{array}
[c]{c}%
w\in V:\text{ }\forall\theta\in\mathbb{M}\text{,}\\%
\begin{array}
[c]{l}%
\left\langle f_{\gamma}\left(  x,u,v\right)  -f_{\gamma}\left(  y,u,w\right)
,x-y\right\rangle \leq0,\\
\left\vert x+g_{\gamma}\left(  \theta,x,u,v\right)  -y-g_{\gamma}\left(
\theta,y,u,w\right)  \right\vert \leq\left\vert x-y\right\vert ,\text{ }\\
\left\vert h\left(  \gamma,x,u,v\right)  -h\left(  \gamma,y,u,w\right)
\right\vert \leq Lip\left(  h\right)  \left\vert x-y\right\vert ,
\end{array}
\end{array}
\right\}  .
\end{align*}
One easily checks that the function has compact values and is upper
semicontinuous. Hence, there exists some measurable selection
\[
\widehat{w}:\mathbb{M}\times\mathbb{K}^{2}\times U\times V\longrightarrow
V,\text{ }\widehat{w}\left(  \gamma,x,y,u,v\right)  \in\Xi\left(
\gamma,x,y,u,v\right)  ,
\]
for all $\left(  \gamma,x,y,u,v\right)  \in\mathbb{M}\times\mathbb{K}%
^{2}\times U\times V.$ For further details, the reader is referred to
\cite[Subsection 9.2]{aubin_frankowska_90}.

We construct an $\mathbb{F}$-predictable $V$-valued control process $w$ as
follows. We begin by fixing $T>0$ (depending on $n$) and $m\geq1$ (depending
on $n$). The choice of $T$ and $m$ will be made explicit later on. Moreover,
we assume that $\mathbb{K\subset}\overline{B}\left(  0,k_{0}\right)  $, for
some $k_{0}>0.$

\underline{Step 1.} \ We consider%
\[
w_{s}^{1,0}:=\widehat{w}\left(  \gamma_{0},x_{0},y_{0},u_{1}\left(  0\right)
,v_{1}\left(  0\right)  \right)  =w_{0}^{1,0},\text{ }s\geq0,
\]
where we have denoted, by abuse of notation,%
\[
u_{1}\left(  s\right)  =u_{1}\left(  \gamma_{0},x_{0},s\right)  ,\text{ }%
v_{1}\left(  s\right)  =v_{1}\left(  \gamma_{0},x_{0},s\right)  .
\]
We recall that if $s\leq\frac{1}{n},$ one has $u_{1}\left(  s\right)
=u_{1}\left(  0,\gamma_{0},x_{0}\right)  $ and similar assertions hold true
for $v_{1}.$ By recalling that $\mathbb{K}$ is invariant with respect to the
controlled piecewise deterministic dynamics, one gets%
\begin{align}
&  \left\langle f_{\gamma_{0}}\left(  \Phi_{s}^{0,x_{0},u,v;\gamma_{0}}%
,u_{s},v_{s}\right)  -f_{\gamma_{0}}\left(  \Phi_{s}^{0,y_{0},u,w^{1,0}%
;\gamma_{0}},u_{s},w_{s}^{1,0}\right)  ,\Phi_{s}^{0,x_{0},u,v;\gamma_{0}}%
-\Phi_{s}^{0,y_{0},u,w^{1,0};\gamma_{0}}\right\rangle \nonumber\\
&  \leq\left\langle f_{\gamma_{0}}\left(  x_{0},u_{s},v_{s}\right)
-f_{\gamma_{0}}\left(  y_{0},u_{s},w_{s}^{1,0}\right)  ,x_{0}-y_{0}%
\right\rangle +c\left(  \left\vert x_{0}-\Phi_{s}^{0,x_{0},u,v;\gamma_{0}%
}\right\vert +\left\vert y_{0}-\Phi_{s}^{0,y_{0},u,w^{1,0};\gamma_{0}%
}\right\vert \right) \nonumber\\
&  \leq c\left(  \left\vert x_{0}-\Phi_{s}^{0,x_{0},u,v;\gamma_{0}}\right\vert
+\left\vert y-\Phi_{s}^{0,y_{0},u,w^{1,0};\gamma_{0}}\right\vert \right)
\leq\frac{c}{n}, \label{ineqf1}%
\end{align}
for all $0\leq s\leq\frac{1}{n}$. Similarly,
\begin{equation}%
\begin{array}
[c]{l}%
i.\text{ \ \ }\frac{c}{n}\geq\left\vert h\left(  \gamma_{0}\Phi_{s}%
^{0,x_{0},u,v;\gamma_{0}},u_{s},v_{s}\right)  -h\left(  \gamma_{0}\Phi
_{s}^{0,y_{0},u,w^{1,0};\gamma_{0}},u_{s},w_{s}^{1,0}\right)  \right\vert \\
\text{ \ \ \ \ \ \ \ \ }-Lip\left(  h\right)  \left\vert \Phi_{s}%
^{0,x_{0},u,v;\gamma_{0}}-\Phi_{s}^{0,y_{0},u,w^{1,0};\gamma_{0}}\right\vert
,\\
ii.\text{ \ }\frac{c}{n}\geq\left\vert \Phi_{s}^{0,x_{0},u,v;\gamma_{0}%
}+g_{\gamma_{0}}\left(  \theta,\Phi_{s}^{0,x_{0},u,v;\gamma_{0}},u_{s}%
,v_{s}\right)  -\Phi_{s}^{0,y_{0},u,w^{1,0};\gamma_{0}}-g_{\gamma_{0}}\left(
\theta,\Phi_{s}^{0,y_{0},u,w^{1,0};\gamma_{0}},u_{s},w_{s}^{1,0}\right)
\right\vert \\
\text{ \ \ \ \ \ \ \ }-\left\vert \Phi_{s}^{0,x_{0},u,v;\gamma_{0}}-\Phi
_{s}^{0,y_{0},u,w^{1,0};\gamma_{0}}\right\vert ,
\end{array}
\label{ineqhg1}%
\end{equation}
for all $0\leq s\leq\frac{1}{n}.$ The constant $c>1$ is generic, independent
of $\delta>0,n,s,x,y,u$ and is allowed to change from one line to another. We
define the control process $w^{1,1}$ by setting%
\[
w_{s}^{1,1}=w_{s}^{1,0}1_{s\leq\frac{1}{n}}+\widehat{w}\left(  \gamma_{0}%
,\Phi_{\frac{1}{n}}^{0,x_{0},u,v;\gamma_{0}},\Phi_{\frac{1}{n}}^{0,y_{0}%
,u,w^{1,0};\gamma_{0}},u_{1}\left(  \frac{1}{n}\right)  ,v_{1}\left(  \frac
{1}{n}\right)  \right)  1_{s>\frac{1}{n}}.
\]
Then the estimates in (\ref{ineqf1}) hold true for $s\in\left[  0,\frac{2}%
{n}\right]  $ if substituting $w^{1,1}$ to $w^{1,0}$. We set
\[
w_{s}^{1,2}=w_{s}^{1,1}1_{s\leq\frac{2}{n}}+\widehat{w}\left(  \gamma_{0}%
,\Phi_{\frac{2}{n}}^{0,x_{0},u,v;\gamma_{0}},\Phi_{\frac{2}{n}}^{0,y_{0}%
,u,w^{1,1};\gamma_{0}},u_{1}\left(  \frac{2}{n}\right)  ,v_{1}\left(  \frac
{2}{n}\right)  \right)  1_{s>\frac{2}{n}},
\]
and so on, to define $w^{1,3},w^{1,4},...,w^{1,n\left(  \left[  T\right]
+1\right)  }$. \ We fix some $w^{0}\in V$ and let $w^{1}=w_{s}^{1,n\left(
\left[  T\right]  +1\right)  }1_{s\leq T}+w^{0}1_{s>T},$ where $\left[
\cdot\right]  $ denotes the floor function. As consequence, by recalling that,
prior to $T_{1}$, both $X^{\gamma_{0},x_{0},u,v}$ and $Y^{y_{0},w}$ are
deterministic and can be identified with $\Phi^{0,x_{0},u,v;\gamma_{0}}$
(resp. $\Phi_{s}^{0,y_{0},u,w^{1,0};\gamma_{0}}$), we use (\ref{ineqf1}) to
get
\begin{equation}
\left\vert X_{t}^{\gamma_{0},x_{0},u,v}-Y_{t}^{y_{0},w^{1}}\right\vert
^{2}\leq\left\vert x_{0}-y_{0}\right\vert ^{2}+\frac{ct}{n}, \label{estimX_Y1}%
\end{equation}
for all $t\leq T$ on $t<T_{1}.$ It follows that%
\begin{equation}
\left\vert X_{t}^{\gamma_{0},x_{0},u,v}-Y_{t}^{y_{0},w^{1}}\right\vert
\leq\left\vert x_{0}-y_{0}\right\vert +\frac{c}{2\sqrt{n}}+\frac{t}{\sqrt{n}},
\label{estimX-Y1^1}%
\end{equation}
or all $t\leq T$ on $t<T_{1}.$ Moreover, on $T_{1}\leq T,$ using
(\ref{ineqhg1}.ii) and (\ref{estimX_Y1}) and recalling that $\mathbb{K}%
\subset\overline{B}\left(  0,k_{0}\right)  $ is invariant (see also Remark
\ref{RemInvarianceY}) one has
\begin{align}
\left\vert X_{T_{1}}^{\gamma_{0},x_{0},u,v}-Y_{T_{1}}^{y_{0},w^{1}%
}\right\vert  &  =\left\vert \Phi_{T_{1}-}^{0,x_{0},u,v;\gamma_{0}}%
+g_{\gamma_{0}}\left(  \gamma_{T_{1}}^{\gamma_{0},x_{0},u,v},\Phi_{T_{1}%
}^{0,x_{0},u,v;\gamma_{0}},u_{T_{1}},v_{T_{1}}\right)  \right. \nonumber\\
&  \left.  -\Phi_{T_{1}-}^{0,y_{0},u,w^{1,0};\gamma_{0}}-g_{\gamma_{0}}\left(
\gamma_{T_{1}}^{\gamma_{0},x_{0},u,v},\Phi_{T_{1}}^{0,y_{0},u,w^{1,0}%
;\gamma_{0}},u_{T_{1}},w_{T_{1}}^{1,0}\right)  \right\vert \nonumber\\
&  \leq\left\vert \Phi_{T_{1}-}^{0,x_{0},u,v;\gamma_{0}}-\Phi_{T_{1}%
-}^{0,y_{0},u,w^{1,0};\gamma_{0}}\right\vert +\frac{c}{n}\nonumber\\
&  \leq\min\left(  \sqrt{\left\vert x_{0}-y_{0}\right\vert ^{2}+\frac{cT_{1}%
}{n}},2k_{0}\right)  +\frac{c}{n}\nonumber\\
&  \leq\left\vert x_{0}-y_{0}\right\vert +\frac{c+T_{1}}{\sqrt{n}},
\label{estimX-YT1^1}%
\end{align}
for all $n\geq4.$ (The reader is invited to recall that $u_{T_{1}}$ is still
$u_{1}\left(  T_{1},\gamma_{0},x_{0}\right)  $ and that $u_{1}\in
\mathcal{A}_{0}^{n}$ is left continuous). One gets, on $T_{1}\leq T,$
\begin{align}
\left\vert X_{T_{1}}^{\gamma_{0},x_{0},u,v}-Y_{T_{1}}^{y_{0},w^{1}}\right\vert
^{2}  &  \leq\left(  \min\left(  \sqrt{\left\vert x_{0}-y_{0}\right\vert
^{2}+\frac{cT_{1}}{n}},2k_{0}\right)  +\frac{c}{n}\right)  ^{2}\leq\left\vert
x_{0}-y_{0}\right\vert ^{2}+\frac{cT_{1}}{n}+\frac{c^{2}}{n^{2}}+4k_{0}%
\frac{c}{n}\nonumber\\
&  \leq\left\vert x_{0}-y_{0}\right\vert ^{2}+\frac{c\left(  T_{1}%
+4k_{0}+1\right)  }{n}, \label{estimX-YT1}%
\end{align}
whenever $n\geq c.$ Finally, using (\ref{ineqhg1}.i) and (\ref{estimX-Y1^1}),
we get%
\begin{align*}
\left\vert h\left(  \gamma_{0},\Phi_{s}^{0,x_{0},u,v;\gamma_{0}},u_{s}%
,v_{s}\right)  -h\left(  \gamma_{0},\Phi_{s}^{0,y_{0},u,w^{1,0};\gamma_{0}%
},u_{s},w_{s}^{1}\right)  \right\vert  &  \leq Lip\left(  h\right)  \left\vert
\Phi_{s}^{0,x_{0},u,v;\gamma_{0}}-\Phi_{s}^{0,y_{0},u,w^{1,0};\gamma_{0}%
}\right\vert +\frac{c}{n}\\
&  \leq Lip\left(  h\right)  \left\vert x_{0}-y_{0}\right\vert +\frac{\left(
Lip\left(  h\right)  +1\right)  c+Lip\left(  h\right)  t}{\sqrt{n}}.
\end{align*}
for all $s<T\wedge T_{1}.$

\underline{Step 2.} We continue the construction on $\left[  T_{1}%
,T_{2}\right)  .$ By abuse of notation, we let $u_{2}\left(  s\right)
:=u_{2}\left(  \gamma_{T_{1}}^{\gamma_{0},x_{0},u,v},X_{T_{1}}^{\gamma
_{0},x_{0},u,v},s\right)  $, for $s\geq0$ and similar for $v_{2}$. We set%
\[
w_{s}^{2,1}:=w_{s}^{1}1_{0\leq s\leq T_{1}}+\widehat{w}\left(  \gamma_{T_{1}%
}^{\gamma_{0},x_{0},u,v},X_{T_{1}}^{\gamma_{0},x_{0},u,v},Y_{T_{1}}%
^{y_{0},w^{1}},u_{2}\left(  0\right)  ,v_{2}\left(  0\right)  \right)
1_{s>T_{1}}.
\]
It is clear that this control process is predictable. We apply the same method
as in Step 1 ($\omega$-wise) on the (stochastic) time interval $\left[
T_{1},\left(  T_{1}+\frac{1}{n}\right)  \wedge T_{2}\right]  ,$ then on
$\left[  T_{1},\left(  T_{1}+\frac{2}{n}\right)  \wedge T_{2}\right]  $, etc.
We construct a sequence of control processes $\left(  w^{2,m}\right)
_{m\geq0}$ and, by choosing $m$ large enough, we establish the existence of
some $w^{2}$ such that
\begin{equation}
\left\{
\begin{array}
[c]{l}%
\left\vert X_{t}^{\gamma_{0},x_{0},u,v}-Y_{t}^{y_{0},w^{2}}\right\vert
^{2}\leq\left\vert X_{T_{1}}^{\gamma_{0},x_{0},u,v}-Y_{T_{1}}^{y_{0},w^{1}%
}\right\vert ^{2}+\frac{c\left(  t-T_{1}\right)  }{n}\leq\left\vert
x_{0}-y_{0}\right\vert ^{2}+\frac{c\left(  t+4k_{0}+1\right)  }{n}\\
\left\vert X_{T_{2}}^{\gamma_{0},x_{0},u,v}-Y_{T_{2}}^{y_{0},w^{2}}\right\vert
\leq\left\vert X_{T_{2}-}^{\gamma_{0},x_{0},u,v}-Y_{T_{2}-}^{y_{0},w^{1}%
}\right\vert +\frac{c}{n}\\
\text{ \ \ \ \ \ \ \ \ \ \ \ \ \ \ \ \ \ \ \ \ \ \ \ \ \ \ \ \ \ }\leq
\min\left(  \sqrt{\left\vert x_{0}-y_{0}\right\vert ^{2}+\frac{c\left(
T_{2}+4k_{0}+1\right)  }{n}},2k_{0}\right)  +\frac{c}{n}\leq\left\vert
x_{0}-y_{0}\right\vert +\frac{c+T_{2}+4k_{0}+1}{\sqrt{n}},\\
\left\vert X_{T_{2}}^{\gamma_{0},x_{0},u,v}-Y_{T_{2}}^{y_{0},w^{2}}\right\vert
^{2}\leq\left\vert x_{0}-y_{0}\right\vert ^{2}+\frac{c\left(  T_{2}%
+4k_{0}+1\right)  }{n}+\frac{c^{2}}{n^{2}}+4k_{0}\frac{c}{n}\leq\left\vert
x_{0}-y_{0}\right\vert ^{2}+\frac{c\left(  T_{2}+2\left(  4k_{0}+1\right)
\right)  }{n},\\
\left\vert h\left(  \gamma_{t}^{\gamma_{0},x_{0},u,v},X_{t}^{\gamma_{0}%
,x_{0},u,v},u_{t},v_{t}\right)  -h\left(  \gamma_{t}^{\gamma_{0},x_{0}%
,u,v},Y_{t}^{y_{0},w^{2}},u_{t},w_{t}^{2}\right)  \right\vert \\
\leq Lip\left(  h\right)  \left\vert X_{t}^{\gamma_{0},x_{0},u,v}-Y_{t}%
^{y_{0},w^{1}}\right\vert +\frac{c}{n}\leq Lip\left(  h\right)  \left\vert
x_{0}-y_{0}\right\vert +\frac{(Lip\left(  h\right)  +1)c+Lip\left(  h\right)
\left(  t+4k_{0}+1\right)  }{\sqrt{n}},
\end{array}
\right.  \label{estimates0.2}%
\end{equation}
for all $T_{1}\leq t<T_{2}\wedge T$ $,$ $\mathbb{P-}a.s.$ if $n\geq\max\left(
4,c\right)  .$ We continue our construction on $\left[  0,T_{3}\wedge
T\right]  ,$ $\left[  0,T_{4}\wedge T\right]  $ and so on to finally get a
predictable process $w^{m}$ such that
\begin{equation}
\left\{
\begin{array}
[c]{l}%
\left\vert X_{t}^{\gamma_{0},x_{0},u,v}-Y_{t}^{y_{0},w^{m}}\right\vert
^{2}\leq\left\vert x_{0}-y_{0}\right\vert ^{2}+\frac{c\left[  t+\left(
i-1\right)  \left(  4k_{0}+1\right)  \right]  }{n},\\
\left\vert X_{T_{i}}^{\gamma_{0},x_{0},u,v}-Y_{T_{i}}^{y_{0},w^{m}}\right\vert
\leq\min\left(  \sqrt{\left\vert x_{0}-y_{0}\right\vert ^{2}+\frac{c\left[
T_{i}+\left(  i-1\right)  \left(  4k_{0}+1\right)  \right]  }{n}}%
,2k_{0}\right)  +\frac{c}{n}\\
\text{ \ \ \ \ \ \ \ \ \ \ \ \ \ \ \ \ \ \ \ \ \ \ \ \ \ \ \ \ \ \ }%
\leq\left\vert x_{0}-y_{0}\right\vert +\frac{c+T_{i}+\left(  i-1\right)
\left(  4k_{0}+1\right)  }{\sqrt{n}},\\
\left\vert X_{T_{i}}^{\gamma_{0},x_{0},u,v}-Y_{T_{i}}^{y_{0},w^{m}}\right\vert
^{2}\leq\left\vert x_{0}-y_{0}\right\vert ^{2}+\frac{c\left[  T_{i}+i\left(
4k_{0}+1\right)  \right]  }{n},\\
\left\vert h\left(  \gamma_{t}^{\gamma_{0},x_{0},u,v},X_{t}^{\gamma_{0}%
,x_{0},u,v},u_{t},v_{t}\right)  -h\left(  \gamma_{t}^{\gamma_{0},x_{0}%
,u,v},Y_{t}^{y_{0},w^{m}},u_{t},w_{t}^{m}\right)  \right\vert \\
\text{ \ \ \ \ \ \ \ \ \ \ \ \ \ \ \ \ \ \ \ \ \ \ \ \ \ \ \ \ \ }\leq
Lip\left(  h\right)  \left\vert x_{0}-y_{0}\right\vert +\frac{\left(
Lip\left(  h\right)  +1\right)  c+Lip\left(  h\right)  \left[  t+\left(
i-1\right)  \left(  4k_{0}+1\right)  \right]  }{\sqrt{n}},
\end{array}
\right.  \label{estimateswm}%
\end{equation}
for all $i\leq m,$ $T_{i-1}\leq t<T_{i}\wedge T,$ $\mathbb{P-}a.s.$

Let us note that in the same way as \cite[Inequality 3.27]{Soner86_2}, one has%
\[
\mathbb{E}\left[  e^{-\delta T_{m}}\right]  \leq\left(  1-\delta\int%
_{0}^{\infty}e^{-t\left(  \delta+\lambda_{\max}\right)  }dt\right)
^{m}=\left(  \frac{\lambda_{\max}}{\delta+\lambda_{\max}}\right)  ^{m},
\]
where $\lambda_{\max}=\underset{\left(  \gamma,x,u,v\right)  \in
\mathbb{M\times R}^{N}\times U\times V}{\sup}\left\vert \lambda_{\gamma
}\left(  x,u,v\right)  \right\vert .$ Then, using the estimates
(\ref{estimateswm}), one gets
\begin{align*}
&  \mathbb{E}\left[  \delta\int_{0}^{\infty}e^{-\delta t}\left\vert h\left(
\gamma_{t}^{\gamma_{0},x_{0},u,v},X_{t}^{\gamma_{0},x_{0},u,v},u_{t}%
,v_{t}\right)  -h\left(  \gamma_{t}^{\gamma_{0},x_{0},u,v},Y_{t}^{y_{0},w^{m}%
},u_{t},w_{t}^{m}\right)  \right\vert dt\right]  \\
&  \leq\mathbb{E}\left[
\begin{array}
[c]{c}%
\delta\sum_{i=0}^{m-1}\int_{T_{i}\wedge T}^{T_{i+1}\wedge T}e^{-\delta
t}\left(  Lip\left(  h\right)  \left\vert x_{0}-y_{0}\right\vert
+\frac{\left(  Lip\left(  h\right)  +1\right)  c+Lip\left(  h\right)  \left[
t+i\left(  4k_{0}+1\right)  \right]  }{\sqrt{n}}\right)  dt\\
+2h_{\max}1_{T_{m}<T}\int_{T_{m}}^{T}\delta e^{-\delta t}dt+2h_{\max
}e^{-\delta T}%
\end{array}
\right]  \\
&  \leq Lip\left(  h\right)  \left\vert x_{0}-y_{0}\right\vert +\frac{\left(
Lip\left(  h\right)  +1\right)  c+mLip\left(  h\right)  \left(  4k_{0}%
+1\right)  }{\sqrt{n}}+\frac{Lip\left(  h\right)  }{\sqrt{n}}\int_{0}^{\infty
}\delta e^{-\delta t}tdt+2h_{\max}e^{-\delta T}\\
&  +2h_{\max}\mathbb{E}\left[  \left(  e^{-\delta T_{m}}-e^{-\delta T}\right)
1_{T_{m}<T}\right]  \\
&  \leq Lip\left(  h\right)  \left\vert x_{0}-y_{0}\right\vert +\frac
{Lip\left(  h\right)  +1}{\sqrt{n}}\left(  c+m\left(  4k_{0}+1\right)
+\frac{1}{\delta}\right)  +2h_{\max}e^{-\delta T}+2h_{\max}\mathbb{E}\left[
e^{-\delta T_{m}}\right]  \\
&  \leq Lip\left(  h\right)  \left\vert x_{0}-y_{0}\right\vert +\frac
{Lip\left(  h\right)  +1}{\sqrt{n}}\left(  c+m\left(  4k_{0}+1\right)
+\frac{1}{\delta}\right)  +2h_{\max}e^{-\delta T}+2h_{\max}\left(
\frac{\lambda_{\max}}{\delta+\lambda_{\max}}\right)  ^{m}.
\end{align*}
The proof of our Lemma is now complete by picking $T=m=n^{\frac{1}{4}}$.
\end{proof}

\section{Appendix\label{Section7Appendix}}

We provide, in this appendix, the key elements of proof leading to Theorem
\ref{ThPiecewiseCtControl}. As we have already hinted before, the proof relies
on the same arguments as those developed in \cite{Krylov_step_99} combined
with dynamic programming principles developed in \cite{Soner86_2}.

If one assumes that $\mathbb{K\subset}\overline{B}\left(  0,k_{0}\right)  $ is
convex and invariant w.r.t. the PDMP governed by $\left(  f,\lambda,Q\right)
$, then one modifies the dynamics such that for $\rho\in\left\{
f,\lambda\right\}  $, $\widetilde{\rho}_{\gamma}\left(  x,u\right)  =0,$ if
$x\notin B\left(  0,k_{0}+1\right)  ,$ $\widetilde{\rho}_{\gamma}\left(
x,u\right)  =\rho_{\gamma}\left(  x,u\right)  ,$ if $x\in\mathbb{K}$ and
setting, for example, $\widetilde{g}_{\gamma}\left(  \theta,x,u\right)
=\Pi_{\mathbb{K}}\left(  x\right)  -x+g_{\gamma}\left(  \theta,\Pi
_{\mathbb{K}}\left(  x\right)  ,u\right)  ,$ for all $x\in%
\mathbb{R}
^{N}.$ Here, $\Pi_{\mathbb{K}}$ stands for the projector onto $\mathbb{K}$. In
this way, all jumps $x\mapsto x+\widetilde{g}\left(  \theta,x,u\right)
=\Pi_{K}\left(  x\right)  +g\left(  \theta,\Pi_{K}\left(  x\right)  ,u\right)
$ take the trajectory in $\mathbb{K}$ (by invariance of this set) and, if the
trajectory reaches $\overline{B}\left(  0,k_{0}+1\right)  ,$ it stays
constant. For the extended dynamics (constructed from this modification as in
Subsection \ref{Subsection6.1Krylov}), one gets
\begin{align*}
\overline{f}_{\gamma}\left(  x,u^{1},u^{2}\right)   &  =\widetilde{f}_{\gamma
}\left(  x+u^{2},u^{1}\right)  =0,\text{ for all }x\in%
\mathbb{R}
^{N}\text{ such that }\left\vert x\right\vert \geq k_{0}+2,\text{ }\\
x+\overline{g}_{\gamma}\left(  \theta,x,u^{1},u^{2}\right)   &  =\left[
x+u^{2}+\widetilde{g}_{\gamma}\left(  \theta,x+u^{2},u^{1}\right)  \right]
-u^{2}\in\mathbb{K-}u^{2}\subset\overline{B}\left(  0,k_{0}+1\right)  ,\text{
for all }x\in%
\mathbb{R}
^{N},
\end{align*}
for all $u^{1}\in\mathbb{U}$ and all $\left\vert u^{2}\right\vert \leq1.$

It follows that the set $\mathbb{K}^{+}:=\overline{B}\left(  0,k_{0}+2\right)
$ is invariant w.r.t the extended dynamics. In fact all sets $\overline
{B}\left(  0,k_{0}+n\right)  $, $n\geq2$ are invariant. Let us emphasize that
this construction is the only point in which the convexity of $\mathbb{K}$
plays a part and it can be avoided by further assumptions.

Let us fix, for the time being, $\delta>0$ and $n\geq1.$

\subsection{Dynamic Programming Principle(s) for (Time) Piecewise Constant
Policies}

The first ingredient is to provide dynamic programming principles and uniform
continuity for the value functions given with respect to piecewise constant
policies with respect to the initial and auxiliary systems (cf. Subsection
\ref{Subsection6.1Krylov}). In addition to $\mathcal{A}_{0}^{n}=\mathcal{A}%
_{0}^{n}\left(  \mathbb{U}\right)  $ and $\mathcal{A}_{ad}^{n}:=\mathcal{A}%
_{ad}^{n}\left(  \mathbb{U}\right)  ,$ one introduces $\mathcal{B}_{0}%
^{n}=\mathcal{A}_{0}^{n}\left(  \mathbb{U\times}\overline{B}\left(
0,1\right)  \right)  $ and $\mathcal{B}_{ad}^{n}=\mathcal{A}_{ad}^{n}\left(
\mathbb{U\times}\overline{B}\left(  0,1\right)  \right)  $ and
\[
v_{\varepsilon}^{\delta,n}(\gamma,x)=\inf_{\left(  u^{1},u^{2}\right)
\in\mathcal{B}_{ad}^{n}}\delta\mathbb{E}\left[  \int_{0}^{\infty}e^{-\delta
t}h\left(  \Gamma_{t}^{\gamma,x,u^{1},\varepsilon u^{2}},X_{t}^{\gamma
,x,u^{1},\varepsilon u^{2}}+\varepsilon u_{t}^{2},u_{t}^{1}\right)  dt\right]
,
\]
for all initial data $\gamma\in\mathbb{M},$ $x\in\overline{B}\left(
0,k_{0}+3\right)  .$

One begins with proving the dynamic programming principles.
\[
\left.
\begin{array}
[c]{l}%
v^{\delta,n}\left(  \gamma,x\right)  =\underset{u\in\mathcal{A}_{0}^{n}}{\inf
}\mathbb{E}\left[  \int_{0}^{T_{1}\wedge T}\delta e^{-\delta t}h\left(
\Gamma_{t}^{\gamma,x,u,0},X_{t}^{\gamma,x,u,0},u_{t}\right)  dt+e^{-\delta
\left(  T_{1}\wedge T\right)  }v^{\delta,n}\left(  \Gamma_{T_{1}\wedge
T}^{\gamma x,u,0},X_{T_{1}\wedge T}^{\gamma x,u,0}\right)  \right]  \text{ and
}\\
v_{\varepsilon}^{\delta,n}(\gamma,x)=\underset{u\in\mathcal{B}_{0}^{n}}{\inf
}\mathbb{E}\left[
\begin{array}
[c]{c}%
\int_{0}^{T_{1}\wedge T}\delta e^{-\delta t}h\left(  \Gamma_{t}^{\gamma
,x,u^{1},\varepsilon u^{2}},X_{t}^{\gamma,x,u^{1},\varepsilon u^{2}%
}+\varepsilon u_{t}^{2},u_{t}^{1}\right)  dt\\
+e^{-\delta\left(  T_{1}\wedge T\right)  }v_{\varepsilon}^{\delta,n}\left(
\Gamma_{T_{1}\wedge T}^{\gamma,x,u^{1},\varepsilon u^{2}},X_{T_{1}\wedge
T}^{\gamma,x,u^{1},\varepsilon u^{2}}\right)
\end{array}
\right]  .
\end{array}
\right.
\]
The arguments are similar to those employed in \cite{Soner86_2}. We will only
emphasize what changes when using controls from $\mathcal{A}_{ad}^{n}$ (or
$\mathcal{B}_{ad}^{n}$) instead of the (more) classical $\mathcal{A}_{ad}.$

Following \cite{Soner86_2}, we introduce%
\[
w^{M,n}\left(  \gamma,x\right)  :=\inf_{u\in\mathcal{A}_{0}^{n}}J^{M,n}\left(
\gamma,x,u\right)  ,
\]
where
\[
J^{M,n}\left(  \gamma,x,u\right)  :=\mathbb{E}\left[  \delta\int_{0}^{T_{1}%
}e^{-\delta t}h\left(  \Gamma_{t}^{\gamma,x,u,0},X_{t}^{\gamma,x,u,0}%
,u_{t}\right)  dt+e^{-\delta T_{1}}w^{M-1,n}\left(  \Gamma_{T_{1}}%
^{\gamma,x,u,0},X_{T_{1}}^{\gamma,x,u,0}\right)  \right]  ,
\]
whenever $M\geq1$. The initial value $w^{0,n}$ is given with respect to the
deterministic control problem (with no jump) and it is standard to check that
it is H\"{o}lder continuous (the H\"{o}lder exponent may be chosen
$\frac{\delta}{Lip\left(  f\right)  }$, where $Lip\left(  f\right)  $ is the
Lipschitz constant for $f_{\gamma}$ for all $\gamma\in\mathbb{M}$ and the
H\"{o}lder constant only depends on the Lipschitz constants and supremum norm
of $f$ and $h$). In particular, the continuity modulus of $w^{0,n}$ (resp.
$w_{\varepsilon}^{0,n}$ defined w.r.t. $\mathcal{B}_{0}^{n}$) is independent
of $n$ (resp. $n$ and $\varepsilon$).

\underline{Step 1.} If $w^{M-1,n}\in BUC\left(  \mathbb{M\times}%
\mathbb{R}
^{N};%
\mathbb{R}
\right)  ,$ then the dynamic programming principle holds true for $w^{M,n}$
and all $\left(  \gamma,x\right)  \in\mathbb{M}\times%
\mathbb{R}
^{N},$ $T\in n^{-1}%
\mathbb{N}
$ :%
\[
w^{M,n}(\gamma,x)=\inf_{u\in\mathcal{A}_{0}^{n}}\mathbb{E}\left[
\begin{array}
[c]{c}%
\delta\int_{0}^{T_{1}\wedge T}e^{-\delta t}h\left(  \Gamma_{t}^{\gamma
,x,u,0},X_{t}^{\gamma,x,u,0},u_{t}\right)  dt+e^{-\delta T}w^{M,n}\left(
\gamma,\Phi_{T}^{0,x,u;\gamma}\right)  1_{T_{1}>T}\\
+e^{-\delta T_{1}}w^{M-1,n}\left(  \Gamma_{T_{1}}^{\gamma,x,u,0},X_{T_{1}%
}^{\gamma,x,u,0}\right)  1_{\tau_{1}\leq T}%
\end{array}
\right]  .
\]
The proof is identical with the proof of \cite[Lemma 3.1]{Soner86_2}. The
reader needs only note that the control policy given by \cite[Eq.
(3.5)]{Soner86_2} of the form%
\[
\overline{u}\left(  \theta,y,t\right)  :=u\left(  \theta,y,t\right)
1_{\left[  0,T\right]  }\left(  t\right)  +u^{\ast}\left(  \theta,\Phi
_{T}^{0,y,u;\theta},t-T\right)  1_{t>T}%
\]
belongs to $\mathcal{A}_{0}^{n}$ if $u$ and $u^{\ast}$ belong to
$\mathcal{A}_{0}^{n}.$

\underline{Step 2.} Since $\overline{B}\left(  0,k_{0}+3\right)  $ is
invariant with respect to the extended PDMP, one has $w^{M,n}\in BUC\left(
\mathbb{M\times}\overline{B}\left(  0,k_{0}+3\right)  \right)  $ and, for
every $\alpha>0,$ there exists a $\alpha-$optimal control policy $u^{\ast}\in$
$\mathcal{A}_{0}^{n}$ such that
\[
J^{M,n}\left(  \gamma,x,u^{\ast}\right)  \leq w^{M,n}\left(  \gamma,x\right)
+\alpha,
\]
for all $x\in\overline{B}\left(  0,k_{0}+3\right)  .$

Again, the proof is identical with the proof of the analogous Lemma 3.3 in
\cite{Soner86_2} and based on recurrence. The reader needs only note that, for
$r>0,$ there exists a finite family $\left\{  x_{k}:k=1,m\right\}  $ such
that
\[
\overline{B}\left(  0,k_{0}+3\right)  \subset\overset{m}{\underset{k=1}{\cup}%
}B\left(  x_{k},r\right)  .
\]
Then, the control policy $u$ defined after (3.18) in \cite{Soner86_2} belongs
to $\mathcal{A}_{0}^{n}$ if $u_{k}$ belong to $\mathcal{A}_{0}^{n}$, for all
$k=1,m.$ We also wish to point out that the estimates leading to the
continuity modulus of $w^{M,n}$ only depend on the Lipschitz constants and the
supremum of $h$, $f,$ $g$ and $\lambda$ but are independent of the control
policies. In particular, this allows one to work with a common continuity
modulus $\omega^{\delta,M}$ for all $n\geq1$ and $\varepsilon>0.$

One concludes using the same arguments (no particular changes needed) as those
in \cite[Theorem 3.4]{Soner86_2}$.$ Due to \cite[Inequality 3.27]{Soner86_2},
one gets%
\[
\sup_{\left(  \gamma,x\right)  \in\mathbb{M\times}\overline{B}\left(
0,k_{0}+3\right)  }\left\vert v^{\delta,n}\left(  \gamma,x\right)
-w^{M,n}\left(  \gamma,x\right)  \right\vert \leq c\alpha^{M},
\]
where $c>0$ and $0<\alpha<1$ are independent of $n$ ($c$ can be chosen as in
\cite{Soner86_2} equal to $2f_{\max}$ and $\alpha$ as in \cite[Page 1120, last
line]{Soner86_2} to be $1-\delta\int_{0}^{\infty}e^{-\left(  \delta
+\lambda_{\max}\right)  t}dt=\frac{\lambda_{\max}}{\delta+\lambda_{\max}}<1).$
The same is true for $v_{\varepsilon}^{\delta,n}-w_{\varepsilon}^{M,n}$ for
$\varepsilon>0.$ In particular,
\[
\left\vert v_{\varepsilon}^{\delta,n}\left(  \gamma,x\right)  -v_{\varepsilon
}^{\delta,n}\left(  \gamma,y\right)  \right\vert \leq\omega^{\delta,M}\left(
\left\vert x-y\right\vert \right)  +2c\left(  \frac{\lambda_{\max}}%
{\delta+\lambda_{\max}}\right)  ^{M},
\]
i.e. the continuity modulus of $v_{\varepsilon}^{\delta,n}$ can also be chosen
independent of $n\geq1$ and $\varepsilon\geq0$ (we identify $v_{0}^{\delta,n}$
with $v^{\delta,n}$). This common continuity modulus will be denoted by
$\omega^{\delta}$, i.e.
\begin{equation}
\omega^{\delta}\left(  r\right)  =\sup_{n\geq1,\varepsilon>0}\sup
_{\substack{\gamma\in\mathbb{M}\\\left\vert x-y\right\vert \leq r}}\left\vert
v_{\varepsilon}^{\delta,n}\left(  \gamma,x\right)  -v_{\varepsilon}^{\delta
,n}\left(  \gamma,x\right)  \right\vert ,\text{ }r>0,\text{ }\omega^{\delta
}\left(  0\right)  :=\lim_{\substack{r\rightarrow0\\r>0}}\omega^{\delta
}\left(  r\right)  =0. \label{omegadelta}%
\end{equation}
The reader will note that $\omega^{\delta}\left(  r\right)  \geq cr,$ for some
$c>0,$ where the equality corresponds to the Lipschitz case.

\subsection{Estimates and Proof of Theorem \ref{ThPiecewiseCtControl}}

We begin with the following convergence result.

\begin{proposition}
\label{propConvApp}For every $\delta>0$ there exists a decreasing function
$\eta^{\delta}:%
\mathbb{R}
_{+}\longrightarrow%
\mathbb{R}
_{+}$ that satisfies $\lim_{\varepsilon\rightarrow0}\eta^{\delta}\left(
\varepsilon\right)  =0$ and such that
\begin{equation}
\sup_{x\in\mathbb{K}^{+}}\left\vert v_{\varepsilon}^{\delta,n}(\gamma
,x)-v^{\delta,n}\left(  \gamma,x\right)  \right\vert \leq\eta^{\delta}\left(
\varepsilon\right)  ,
\end{equation}
for all $n\geq1$ and all $\varepsilon\geq0.$
\end{proposition}

\begin{proof}
The proof is similar to the one of \cite[Theorem 3.6]{G8}. However, we present
the arguments for reader's sake. Let us fix $\gamma\in\mathbb{M}$,
$x\in\mathbb{K}^{+}$ and $\varepsilon>0.$ The definition of the value
functions implies that $v_{\varepsilon}^{\delta,n}(\gamma,x)\leq v^{\delta
,n}\left(  \gamma,x\right)  $. Standard estimates yield the existence of some
positive constant $C>0$ which is independent of $\gamma,x,$ of $n\geq1$ and
$\varepsilon>0$ such that
\begin{equation}
\left\vert \Phi_{t}^{0,x,u^{1},\varepsilon u^{2};\gamma}-\Phi_{t}%
^{0,x,u^{1},0;\gamma}\right\vert \leq C\varepsilon, \label{est}%
\end{equation}
for all $t\in\left[  0,1\right]  ,$ and all $\left(  u^{1},u^{2}\right)
\in\mathcal{B}_{0}^{n}.$ We recall that $\Phi_{\cdot}^{0,x,u^{1},u^{2};\gamma
}$ is the unique solution of the deterministic equation%
\[
\left\{
\begin{array}
[c]{l}%
d\Phi_{t}^{0,x,u^{1},u^{2};\gamma}=\overline{f}_{\gamma}\left(  \Phi
_{t}^{0,x,u^{1},\varepsilon u^{2};\gamma},u_{t}^{1},u_{t}^{2}\right)
dt=f_{\gamma}\left(  \Phi_{t}^{0,x,u^{1},u^{2};\gamma}+u_{t}^{2},u_{t}%
^{1}\right)  dt,\\
\Phi_{0}^{0,x,u^{1},u^{2};\gamma}=x.
\end{array}
\right.
\]
The constant $C$ in (\ref{est}) is generic and may change from one line to
another. We emphasize that throughout the proof, $C$ may be chosen independent
of $x\in%
\mathbb{R}
^{N},$ $n\geq1,$ $\varepsilon>0$ and of $\left(  u^{1},u^{2}\right)
\in\mathcal{B}_{0}^{n}$ (it only depends on Lipschitz constants and bounds of
$f$, $\lambda,$ $g$ and $h$)$.$ The dynamic programming principle written for
$v^{\delta,n}$ yields
\begin{equation}
v^{\delta,n}(\gamma,x)\leq\mathbb{E}\left[  \int_{0}^{T_{1}\wedge1}\delta
e^{-\delta t}h\left(  \Gamma_{t}^{\gamma,x,u^{1},0},X_{t}^{\gamma,x,u^{1}%
,0},u_{t}^{1}\right)  dt+e^{-\delta\left(  T_{1}\wedge1\right)  }v^{\delta
,n}\left(  \Gamma_{T_{1}\wedge1}^{\gamma,x,u^{1},0},X_{T_{1}\wedge1}%
^{\gamma,x,u^{1},0}\right)  \right]  , \label{I0}%
\end{equation}
for all $u^{1}\in\mathcal{A}_{0}^{n}$. We consider an arbitrary admissible
control couple $\left(  u^{1},u^{2}\right)  \in\mathcal{B}_{0}^{n}$. For
simplicity, we introduce the following notations:
\begin{align*}
u_{t}^{i}  &  =u^{i}\left(  x,t\right)  ,\text{ }i=1,2,\\
\lambda^{1}\left(  t\right)   &  =\lambda_{\gamma}\left(  \Phi_{t}%
^{0,x,u^{1},0;\gamma},u_{t}^{1}\right)  ,\text{ }\Lambda^{1}\left(  t\right)
=\exp\left(  -\int_{0}^{t}\lambda^{1}\left(  s\right)  ds\right) \\
\lambda^{1,2}\left(  t\right)   &  =\lambda_{\gamma}\left(  \Phi
_{t}^{0,x,u^{1},\varepsilon u^{2};\gamma}+\varepsilon u_{t}^{2},u_{t}%
^{1}\right)  ,\text{ }\Lambda^{1,2}\left(  t\right)  =\exp\left(  -\int%
_{0}^{t}\lambda^{1,2}\left(  s\right)  ds\right)  ,
\end{align*}
for all $t\geq0$. We denote the right-hand member of the inequality (\ref{I0})
by $I.$ Then, $I$ is explicitly given by
\begin{align*}
I  &  =\int_{0}^{1}\lambda^{1}(t)\Lambda^{1}\left(  t\right)  \int_{0}%
^{t}\delta e^{-\delta s}h\left(  \gamma,\Phi_{s}^{0,x,u^{1},0;\gamma}%
,u_{s}^{1}\right)  dsdt\\
&  +\int_{0}^{1}\lambda^{1}(t)\Lambda^{1}\left(  t\right)  e^{-\delta t}\int_{%
\mathbb{R}
^{N}}v^{\delta,n}\left(  \theta,\Phi_{t}^{0,x,u^{1},0;\gamma}+g_{\gamma
}\left(  \theta,\Phi_{t}^{0,x,u^{1},0;\gamma},u_{t}^{1}\right)  \right)
Q^{0}\left(  \gamma,u_{t}^{1},d\theta\right)  dt\\
&  +\Lambda^{1}\left(  1\right)  \int_{0}^{1}\delta e^{-\delta t}h\left(
\gamma,\Phi_{t}^{0,x,u^{1},0;\gamma},u_{t}^{1}\right)  dt+\Lambda^{1}\left(
1\right)  e^{-\delta}v^{\delta,n}\left(  \gamma,\Phi_{1}^{0,x,u^{1},0;\gamma
}\right) \\
&  =I_{1}+I_{2}+I_{3}+I_{4}.
\end{align*}
Using the inequality (\ref{est}), one gets%
\begin{align}
I_{1}  &  \leq\int_{0}^{1}\lambda^{1,2}(t)\Lambda^{1,2}\left(  t\right)
\int_{0}^{t}\delta e^{-\delta s}h\left(  \gamma,\Phi_{s}^{0,x,u^{1}%
,\varepsilon u^{2};\gamma}+\varepsilon u_{s}^{2},u_{s}^{1}\right)
dsdt+C\varepsilon,\label{I1}\\
I_{3}  &  \leq\Lambda^{1,2}\left(  1\right)  \int_{0}^{1}\delta e^{-\delta
t}h\left(  \gamma,\Phi_{t}^{0,x,u^{1},\varepsilon u^{2};\gamma}+\varepsilon
u_{t}^{2},u_{t}^{1}\right)  dt+C\varepsilon. \label{I2}%
\end{align}
For the term $I_{2},$ with the notation (\ref{omegadelta}), one has%
\begin{align}
&  I_{2}\nonumber\\
&  \leq\int_{0}^{1}\lambda^{1,2}(t)\Lambda^{1,2}\left(  t\right)  e^{-\delta
t}\int_{%
\mathbb{R}
^{N}}v^{\delta,n}\left(  \theta,\Phi_{t}^{0,x,u^{1},\varepsilon u^{2};\gamma
}+g_{\gamma}\left(  \theta,\Phi_{t}^{0,x,u^{1},\varepsilon u^{2};\gamma
}+\varepsilon u_{t}^{2},u_{t}^{1}\right)  \right)  Q^{0}\left(  \gamma
,u_{t}^{1},d\theta\right)  dt\nonumber\\
&  +C\left(  \varepsilon+\omega^{\delta}\left(  C\varepsilon\right)  \right)
\nonumber\\
&  \leq\int_{0}^{1}\lambda^{1,2}(t)\Lambda^{1,2}\left(  t\right)  e^{-\delta
t}\int_{%
\mathbb{R}
^{N}}v_{\varepsilon}^{\delta,n}\left(  \theta,\Phi_{t}^{0,x,u^{1},\varepsilon
u^{2};\gamma}+g_{\gamma}\left(  \theta,\Phi_{t}^{0,x,u^{1},\varepsilon
u^{2};\gamma}+\varepsilon u_{t}^{2},u_{t}^{1}\right)  \right)  Q^{0}\left(
\gamma,u_{t}^{1},d\theta\right)  dt\nonumber\\
&  +\left(  \int_{0}^{1}\lambda^{1,2}(t)\Lambda^{1,2}\left(  t\right)
e^{-\delta t}dt\right)  \sup_{\theta\in\mathbb{M},z\in\mathbb{K}^{+}%
}\left\vert v^{\delta,n}(z)-v_{\varepsilon}^{\delta,n}(z)\right\vert +C\left(
\varepsilon+\omega^{\delta}\left(  C\varepsilon\right)  \right)  . \label{I3}%
\end{align}
Finally,
\begin{align}
I_{4}  &  \leq\Lambda^{1,2}\left(  1\right)  e^{-\delta}v^{\delta,n}\left(
\gamma,\Phi_{1}^{0,x,u^{1},\varepsilon u^{2};\gamma}\right)  +C\left(
\varepsilon+\omega^{\delta}\left(  C\varepsilon\right)  \right) \nonumber\\
&  \leq\Lambda^{1,2}(1)e^{-\delta}v_{\varepsilon}^{\delta,n}\left(
\gamma,\Phi_{1}^{0,x,u^{1},\varepsilon u^{2};\gamma}\right)  +\Lambda
^{1,2}(1)e^{-\delta}\sup_{\theta\in\mathbb{M},z\in\mathbb{K}^{+}}\left\vert
v^{\delta,n}(\theta,z)-v_{\varepsilon}^{\delta,n}(\theta,z)\right\vert
+C\left(  \omega^{\delta}\left(  C\varepsilon\right)  +\varepsilon\right)  .
\label{I4}%
\end{align}
We substitute (\ref{I1})-(\ref{I4}) in (\ref{I0}). We take the infimum over
the family of $\left(  u^{1},u^{2}\right)  \in\mathcal{B}_{0}^{n}$ and use the
dynamic programming principle to have
\begin{align*}
v^{\delta,n}\left(  \gamma,x\right)   &  \leq v_{\varepsilon}^{\delta
,n}(\gamma,x)+C\left(  \varepsilon+\omega^{\delta}\left(  C\varepsilon\right)
\right) \\
&  +\left(  \int_{0}^{1}\lambda^{1,2}(t)\Lambda^{1,2}\left(  t\right)
e^{-\delta t}dt+\Lambda^{1,2}(1)e^{-\delta}\right)  \sup_{\theta\in
\mathbb{M},z\in\mathbb{K}^{+}}\left\vert v^{\delta,n}(z)-v_{\varepsilon
}^{\delta,n}(z)\right\vert .
\end{align*}
We notice that
\begin{align*}
\int_{0}^{1}\lambda^{1,2}(t)\Lambda^{1,2}\left(  t\right)  e^{-\delta
t}dt+\Lambda^{1,2}(1)e^{-\delta}  &  =1-\delta\int_{0}^{1}e^{-\int_{0}%
^{t}\overline{\lambda}_{\gamma}\left(  \Phi_{s}^{0,x,u^{1},\varepsilon
u^{2};\gamma},u_{s}^{1},\varepsilon u_{s}^{2}\right)  ds}e^{-\delta t}dt\\
&  \leq\frac{\lambda_{\max}}{\lambda_{\max}+\delta}+\frac{\delta}%
{\lambda_{\max}+\delta}e^{-\left(  \lambda_{\max}+\delta\right)  }.
\end{align*}
Thus,
\begin{align*}
v^{\delta,n}\left(  \gamma,x\right)  -v_{\varepsilon}^{\delta,n}(\gamma,x)  &
\leq C\left(  \varepsilon+\omega^{\delta}\left(  C\varepsilon\right)  \right)
\\
&  +\left(  \frac{\lambda_{\max}}{\lambda_{\max}+\delta}+\frac{\delta}%
{\lambda_{\max}+\delta}e^{-\left(  \lambda_{\max}+\delta\right)  }\right)
\sup_{\theta\in\mathbb{M},z\in%
\mathbb{R}
^{N}}\left\vert v^{\delta,n}(\theta,z)-v_{\varepsilon}^{\delta,n}%
(\theta,z)\right\vert .
\end{align*}
Here, $\lambda_{\max}:=\sup\left\{  \lambda_{\gamma}\left(  x,u\right)
:\gamma\in\mathbb{M},x\in%
\mathbb{R}
^{N},u\in\mathbb{U}\right\}  <\infty.$ The conclusion follows by taking the
supremum over $\theta\in\mathbb{M}$ and $x\in\mathbb{K}^{+}$ and recalling
that $C$ is independent of $x$ and $\varepsilon>0$ (and $n\geq1$)$.$
\end{proof}

We consider $\left(  \rho_{\varepsilon}\right)  $ a sequence of standard
mollifiers i.e. $\rho_{\varepsilon}\left(  y\right)  =\frac{1}{\varepsilon
^{N}}\rho\left(  \frac{y}{\varepsilon}\right)  ,$ $y\in%
\mathbb{R}
^{N},$ $\varepsilon>0,$ where $\rho\in C^{\infty}\left(
\mathbb{R}
^{N}\right)  $ is a positive function such that
\[
Supp(\rho)\subset\overline{B}\left(  0,1\right)  \text{ and }\int_{%
\mathbb{R}
^{N}}\rho(x)dx=1.
\]
We introduce the convoluted functions
\[
v_{(\varepsilon)}^{\delta,n}\left(  \gamma,\cdot\right)  :=v_{\varepsilon
}^{\delta,n}\left(  \gamma,\cdot\right)  \ast\rho_{\varepsilon}.
\]
In analogy to \cite[Lemma 3.5]{Krylov_step_99}, one gets

\begin{proposition}
\label{PropEstimatesVdeltanEps}The value functions $v_{(\varepsilon)}%
^{\delta,n}$ are such that
\[
\left\{
\begin{array}
[c]{c}%
\underset{x\in\mathbb{K}}{\sup}\left(  \left\vert v_{(\varepsilon)}^{\delta
,n}\left(  \gamma,x\right)  \right\vert +\left\vert \partial_{x^{i}%
}v_{(\varepsilon)}^{\delta,n}\left(  \gamma,x\right)  \right\vert \right)
\leq C^{\delta}\varepsilon^{-1}\left(  \varepsilon+\omega^{\delta}\left(
\varepsilon\right)  \right)  ,\\
\underset{_{x,y\in\mathbb{K},\text{ }y\neq x}}{\sup}\frac{\left\vert
\partial_{x^{i}}v_{(\varepsilon)}^{\delta,n}\left(  \gamma,x\right)
-\partial_{x^{i}}v_{(\varepsilon)}^{\delta,n}\left(  \gamma,y\right)
\right\vert }{\left\vert x-y\right\vert }\leq C^{\delta}\varepsilon^{-1}%
\omega^{\delta}\left(  \left\vert x-y\right\vert \right)  \text{ }\\
\underset{x\in\mathbb{K}}{\sup}\left\vert v_{(\varepsilon)}^{\delta,n}\left(
\gamma,x\right)  -v^{\delta,n}\left(  \gamma,x\right)  \right\vert \leq
\omega^{\delta}\left(  \varepsilon\right)  +\eta^{\delta}\left(
\varepsilon\right)  ,
\end{array}
\right.
\]
for all $\gamma\in\mathbb{M}$. Here, $C^{\delta}$ is a positive constant
independent of $\varepsilon>0,$ $n\geq1$ and $\gamma\in\mathbb{M}.$
\end{proposition}

\begin{proof}
To prove the first inequality, one recalls the definition of $v_{(\varepsilon
)}^{\delta,n}$. Then, due to Proposition \ref{propConvApp} and using the
notation (\ref{omegadelta}), one gets%
\begin{align*}
\left\vert \partial_{x^{i}}v_{(\varepsilon)}^{\delta,n}\left(  x\right)
\right\vert  &  =\left\vert \varepsilon^{-1}\int_{\overline{B}\left(
0,1\right)  }v_{\varepsilon}^{\delta,n}\left(  x-\varepsilon y\right)
\partial_{x^{i}}\rho\left(  y\right)  dy\right\vert =\left\vert \varepsilon
^{-1}\underset{y\in\overline{B}\left(  0,1\right)  }{\int}\left(
v_{\varepsilon}^{\delta,n}\left(  x-\varepsilon y\right)  -v_{\varepsilon
}^{\delta,n}\left(  x\right)  \right)  \partial_{x^{i}}\rho\left(  y\right)
dy\right\vert \\
&  \leq C^{\delta}\varepsilon^{-1}\omega^{\delta}\left(  \varepsilon\right)  .
\end{align*}
Similarly,%
\[
\left\vert \partial_{x^{i}}v_{(\varepsilon)}^{\delta,n}\left(  x\right)
-\partial_{x^{i}}v_{(\varepsilon)}^{\delta,n}\left(  y\right)  \right\vert
=\left\vert \varepsilon^{-1}\int_{\overline{B}\left(  0,1\right)  }\left(
v_{\varepsilon}^{\delta,n}\left(  x-\varepsilon z\right)  -v_{\varepsilon
}^{\delta,n}\left(  y-\varepsilon z\right)  \right)  \partial_{x^{i}}%
\rho\left(  z\right)  dy\right\vert \leq C^{\delta}\varepsilon^{-1}%
\omega^{\delta}\left(  \left\vert x-y\right\vert \right)  .
\]
Moreover, again with the notation (\ref{omegadelta}) and the help of
Proposition \ref{propConvApp}, one gets
\begin{align*}
\sup_{x\in%
\mathbb{R}
^{N}}\left\vert v_{(\varepsilon)}^{\delta,n}\left(  x\right)  -v^{\delta
,n}\left(  x\right)  \right\vert  &  =\left\vert \underset{y\in\overline
{B}\left(  0,1\right)  }{\int}\left(  v_{\varepsilon}^{\delta,n}\left(
x-\varepsilon y\right)  -v_{\varepsilon}^{\delta,n}\left(  x\right)
+v_{\varepsilon}^{\delta,n}\left(  x\right)  -v^{\delta,n}\left(  x\right)
\right)  \rho\left(  y\right)  dy\right\vert \\
&  \leq\omega^{\delta}\left(  \varepsilon\right)  +\eta^{\delta}\left(
\varepsilon\right)  .
\end{align*}
The proof of our proposition is now complete.
\end{proof}

We now come to the proof of the main convergence result.

\begin{proof}
(of Theorem \ref{ThPiecewiseCtControl}). Let us fix $\left(  u^{1}%
,u^{2}\right)  \in\mathbb{U}\times\overline{B}\left(  0,1\right)  .$The
dynamic programming principle written for $v_{\varepsilon}^{\delta,n}$ yields
\[
v_{\varepsilon}^{\delta,n}\left(  \gamma,x\right)  \leq\mathbb{E}\left[
\begin{array}
[c]{c}%
\int_{0}^{T_{1}\wedge n^{-1}}\delta e^{-\delta t}h\left(  \Gamma_{t}%
^{\gamma,x,u^{1},\varepsilon u^{2}},X_{t}^{\gamma,x,u^{1},\varepsilon u^{2}%
}+\varepsilon u_{t}^{2},u_{t}^{1}\right)  dt\\
+e^{-\delta\left(  T_{1}\wedge n^{-1}\right)  }v_{\varepsilon}^{\delta
,n}\left(  \Gamma_{T_{1}\wedge n^{-1}}^{\gamma,x,u^{1},\varepsilon u^{2}%
},X_{T_{1}\wedge n^{-1}}^{\gamma x,u^{1},\varepsilon u^{2}}\right)
\end{array}
\right]  ,
\]
where $\left(  u^{1},u^{2}\right)  \in\mathcal{B}_{0}^{n}$ and $x\in
\mathbb{K}$. In particular, if $\left(  u_{t}^{1},u_{t}^{2}\right)  =\left(
u,y\right)  \in\mathbb{U}\times\overline{B}\left(  0,1\right)  ,$ for
$t\in\left[  0,n^{-1}\right)  $, one notices that on $\left[  0,T_{1}\wedge
n^{-1}\right)  ,$
\[
X_{t}^{\gamma,x-\varepsilon y,u^{1},\varepsilon u^{2}}=X_{t}^{\gamma
,x,u^{1},0}-\varepsilon y\text{ and }\Gamma_{t}^{\gamma,x,u^{1},\varepsilon
u^{2}}=\Gamma_{t}^{\gamma,x,u^{1},0}.
\]
As consequence, the (law of the) first jump time starting from $\left(
\gamma,x-\varepsilon y\right)  $ when the trajectory is controlled by the
couple $\left(  u^{1},\varepsilon u^{2}\right)  $ given above only depends on
$u$ and $x$ (but not on $\varepsilon,$ nor on $y$). To emphasize this
dependence, we denote it by $T_{1}^{x,u}$. Similarly, $\left(  \Gamma
_{T_{1}^{x,u}\wedge n^{-1}}^{\gamma,x,u^{1},\varepsilon u^{2}},X_{T_{1}%
^{x,u}\wedge n^{-1}}^{\gamma,x,u^{1},\varepsilon u^{2}}\right)  $ has the same
law as $\left(  \Gamma_{T_{1}^{x,u}\wedge n^{-1}}^{\gamma,x,u^{1},0}%
,X_{T_{1}^{x,u}\wedge n^{-1}}^{\gamma,x,u^{1},0}-\varepsilon y\right)  $. On
gets
\[
v_{\varepsilon}^{\delta,n}\left(  \gamma,x-\varepsilon y\right)
\leq\mathbb{E}\left[
\begin{array}
[c]{c}%
\int_{0}^{T_{1}^{x,u}\wedge n^{-1}}\delta e^{-\delta t}h\left(  \Gamma
_{t}^{\gamma,x,u^{1},0},X_{t}^{\gamma,x,u^{1},0},u_{t}^{1}\right)  dt\\
+e^{-\delta\left(  T_{1}^{x,u}\wedge n^{-1}\right)  }v_{\varepsilon}%
^{\delta,n}\left(  \Gamma_{T_{1}^{x,u}\wedge n^{-1}}^{\gamma,x,u^{1}%
,0},X_{T_{1}^{x,u}\wedge n^{-1}}^{\gamma,x,u^{1},0}-\varepsilon y\right)
\end{array}
\right]  .
\]
It follows that
\begin{align*}
&  v_{\left(  \varepsilon\right)  }^{\delta,n}\left(  \gamma,x\right)  \\
&  \leq\mathbb{E}\left[  \int_{0}^{T_{1}^{x,u}\wedge n^{-1}}\delta e^{-\delta
t}h\left(  \Gamma_{t}^{\gamma,x,u^{1},0},X_{t}^{\gamma,x,u,0},u_{t}%
^{1}\right)  dt+e^{-\delta\left(  T_{1}^{x,u}\wedge n^{-1}\right)  }v_{\left(
\varepsilon\right)  }^{\delta,n}\left(  \Gamma_{T_{1}^{x,u}\wedge n^{-1}%
}^{\gamma,x,u,0},X_{T_{1}^{x,u}\wedge n^{-1}}^{\gamma,x,u,0}\right)  \right]
.
\end{align*}
Applying It\^{o}'s formula to $v_{\left(  \varepsilon\right)  }^{\delta
,n}\left(  \Gamma_{t}^{\gamma,x,u,0},X_{t}^{\gamma,x,u,0}\right)  $ on
$\left[  0,T_{1}^{x,u}\wedge n^{-1}\right]  $ and recalling that $u_{t}^{1}=u$
prior to $n^{-1}$, it follows that%
\begin{align*}
0 &  \leq\mathbb{E}\left[  \int_{0}^{T_{1}^{x,u}\wedge n^{-1}}e^{-\delta
t}\left[
\begin{array}
[c]{c}%
\delta\left(  h\left(  \Gamma_{t}^{\gamma,x,u,0},X_{t}^{\gamma,x,u,0}%
,u\right)  -v_{\left(  \varepsilon\right)  }^{\delta,n}\left(  \Gamma
_{t}^{\gamma,x,u,0},X_{t}^{\gamma x,u,0}\right)  \right)  \\
+\mathcal{L}^{u}v_{\left(  \varepsilon\right)  }^{\delta,n}\left(  \Gamma
_{t}^{\gamma,x,u,0},X_{t}^{\gamma,x,u,0}\right)
\end{array}
\right]  \right]  dt\\
&  =\mathbb{E}\left[  \int_{0}^{T_{1}^{x,u}\wedge n^{-1}}e^{-\delta t}\left[
\begin{array}
[c]{c}%
\delta\left(  h\left(  \gamma,\Phi_{t}^{0,x,u,0;\gamma},u\right)  -v_{\left(
\varepsilon\right)  }^{\delta,n}\left(  \gamma,\Phi_{t}^{0,x,u,0;\gamma
}\right)  \right)  \\
+\mathcal{L}^{u}v_{\left(  \varepsilon\right)  }^{\delta,n}\left(  \gamma
,\Phi_{t}^{0,x,u^{1},0;\gamma}\right)
\end{array}
\right]  \right]  dt\\
&  \leq\mathbb{E}\left[  \int_{0}^{T_{1}^{x,u}\wedge n^{-1}}e^{-\delta
t}dt\right]  \left(  \delta\left(  h\left(  \gamma,x,u\right)  -v_{\left(
\varepsilon\right)  }^{\delta,n}\left(  \gamma,x\right)  \right)
+\mathcal{L}^{u}v_{\left(  \varepsilon\right)  }^{\delta,n}\left(
\gamma,x\right)  \right)  \\
&  +\mathbb{E}\left[  \int_{0}^{T_{1}^{x,u}\wedge n^{-1}}e^{-\delta t}\left[
\begin{array}
[c]{c}%
\delta\left(  \left\vert h\right\vert _{1}+C^{\delta}\varepsilon^{-1}\left(
\varepsilon+\omega^{\delta}\left(  \varepsilon\right)  \right)  \right)
\left\vert f\right\vert _{1}t+\left\vert f\right\vert _{1}^{2}tC^{\delta
}\varepsilon^{-1}\left(  \varepsilon+\omega^{\delta}\left(  \varepsilon
\right)  \right)  \\
+\left\vert f\right\vert _{1}C^{\delta}\varepsilon^{-1}\omega^{\delta}\left(
\frac{\left\vert f\right\vert _{1}}{n}\right)  +2\left\vert h\right\vert
_{1}\left\vert \lambda\right\vert _{1}\left\vert f\right\vert _{1}t\\
+\left\vert \lambda\right\vert _{1}C^{\delta}\varepsilon^{-1}\left(
\varepsilon+\omega^{\delta}\left(  \varepsilon\right)  \right)  (1+\left\vert
g\right\vert _{1})\left\vert f\right\vert _{1}t
\end{array}
\right]  dt\right]  \\
&  \leq\mathbb{E}\left[  T_{1}^{x,u}\wedge n^{-1}\right]  \left(  \left[
\delta\left(  h\left(  \gamma,x,u\right)  -v_{\left(  \varepsilon\right)
}^{\delta,n}\left(  \gamma,x\right)  \right)  +\mathcal{L}^{u}v_{\left(
\varepsilon\right)  }^{\delta,n}\left(  \gamma,x\right)  \right]  \right)  \\
&  +\mathbb{E}\left[  T_{1}^{x,u}\wedge n^{-1}\right]  \widetilde{C}^{\delta
}\left(  \frac{1+\varepsilon^{-1}\omega^{\delta}\left(  \varepsilon\right)
}{n}+\varepsilon^{-1}\omega^{\delta}\left(  \frac{\left\vert f\right\vert
_{1}}{n}\right)  \right)  .
\end{align*}
The generic constant $\widetilde{C}^{\delta}$ is independent of $x\in
\mathbb{K},\gamma\in\mathbb{M},u\in\mathbb{U},\varepsilon>0$ and $n\geq1$ and
may change from one line to another. As consequence,%
\[
\delta\left(  h\left(  \gamma,x,u\right)  -v_{\left(  \varepsilon\right)
}^{\delta,n}\left(  \gamma,x\right)  \right)  +\mathcal{L}^{u}v_{\left(
\varepsilon\right)  }^{\delta,n}\left(  \gamma,x\right)  \geq-\widetilde{C}%
^{\delta}\left(  \frac{1+\varepsilon^{-1}\omega^{\delta}\left(  \varepsilon
\right)  }{n}+\varepsilon^{-1}\omega^{\delta}\left(  \frac{\left\vert
f\right\vert _{1}}{n}\right)  \right)  .
\]
We fix (for the time being), the initial configuration $\left(  \gamma
_{0},x_{0}\right)  \in\mathbb{M\times K}$ and an arbitrary control $u^{1}%
\in\mathcal{A}_{ad}$. We apply the previous inequality for $\left(
\gamma,x\right)  =\left(  \Gamma_{t}^{\gamma_{0},x_{0},u^{1},0},X_{t}%
^{\gamma_{0},x_{0},u^{1},0}\right)  $, integrate the inequality with respect
to $e^{-\delta t}dt$ on $\left[  0,T\right]  $ for $T>0$ and use It\^{o}'s
formula to get%
\begin{align*}
\mathbb{E}\left[  \int_{0}^{T}\delta e^{-\delta t}h\left(  \Gamma_{t}%
^{\gamma_{0},x_{0},u^{1},0},X_{t}^{\gamma_{0},x_{0},u^{1},0},u_{t}^{1}\right)
dt\right]   &  \geq v_{\left(  \varepsilon\right)  }^{\delta,n}\left(
\gamma_{0},x_{0}\right)  -e^{-\delta T}v_{\left(  \varepsilon\right)
}^{\delta,n}\left(  \Gamma_{T}^{\gamma_{0},x_{0},u^{1},0},X_{T}^{\gamma
_{0},x_{0},u^{1},0}\right)  \\
&  -\widetilde{C}^{\delta}\left(  \frac{1+\varepsilon^{-1}\omega^{\delta
}\left(  \varepsilon\right)  }{n}+\varepsilon^{-1}\omega^{\delta}\left(
\frac{\left\vert f\right\vert _{1}}{n}\right)  \right)  .
\end{align*}
One lets $T\rightarrow\infty$ and takes the infimum over $u^{1}\in
\mathcal{A}_{ad}$ to get%
\[
v^{\delta}\left(  \gamma_{0},x_{0}\right)  \geq v_{\left(  \varepsilon\right)
}^{\delta,n}\left(  \gamma_{0},x_{0}\right)  -\widetilde{C}^{\delta}\left(
\frac{1+\varepsilon^{-1}\omega^{\delta}\left(  \varepsilon\right)  }%
{n}+\varepsilon^{-1}\omega^{\delta}\left(  \frac{\left\vert f\right\vert _{1}%
}{n}\right)  \right)  .
\]
Finally, using the third estimate in Proposition \ref{PropEstimatesVdeltanEps}%
, one gets
\[
v^{\delta}\left(  \gamma_{0},x_{0}\right)  \geq v^{\delta,n}\left(  \gamma
_{0},x_{0}\right)  -\widetilde{C}^{\delta}\left(  \frac{1+\varepsilon
^{-1}\omega^{\delta}\left(  \varepsilon\right)  }{n}+\varepsilon^{-1}%
\omega^{\delta}\left(  \frac{\left\vert f\right\vert _{1}}{n}\right)
+\omega^{\delta}\left(  \varepsilon\right)  +\eta^{\delta}\left(
\varepsilon\right)  .\right)  .
\]
The conclusion follows by taking $\varepsilon=\left(  \omega^{\delta}\left(
\frac{\left\vert f\right\vert _{1}}{n}\right)  \right)  ^{1-\eta},$ for some
$1>\eta>0$ (e.g. $\eta=\frac{1}{2}$). Our result is now complete.
\end{proof}

\bibliographystyle{plain}
\bibliography{bibliografie_06092014}

\end{document}